\documentclass[11pt,a4paper]{article}
\newcommand{\TemplateStyle}{manuscript}
\providecommand{\TemplateStyle}{manuscript}
\usepackage{etoolbox}
\newif\ifTemplateManuscript
\newif\ifTemplateProject
\newif\ifTemplateChinese
\ifdefstring{\TemplateStyle}{manuscript}{\TemplateManuscripttrue}{
  \ifdefstring{\TemplateStyle}{project}{\TemplateProjecttrue}{
    \ifdefstring{\TemplateStyle}{chinese}{\TemplateChinesetrue}{
      \PackageError{layout_settings}{Unknown TemplateStyle}{Use manuscript, project, or chinese.}}}}
\providecommand{\ProjectHeader}{Research Plan}
\providecommand{\TemplateLegacyMath}{true}

\usepackage{iftex}
\ifPDFTeX
  \usepackage[utf8]{inputenc}
  \usepackage{CJKutf8}
\else
  \ifXeTeX
    \usepackage{fontspec}
    \usepackage{xeCJK}
    \IfFontExistsTF{DFKai-SB}
      {\setCJKmainfont{DFKai-SB}[AutoFakeBold=2]
       \setCJKsansfont{DFKai-SB}[AutoFakeBold=2]}
      {\IfFontExistsTF{BiauKai}
        {\setCJKmainfont{BiauKai}[AutoFakeBold=2]
         \setCJKsansfont{BiauKai}[AutoFakeBold=2]}
        {
         \IfFileExists{fonts/bkai00mp.ttf}{
           \setCJKmainfont{bkai00mp.ttf}[Path=fonts/,AutoFakeBold=2,AutoFakeSlant=.2]
           \setCJKsansfont{bkai00mp.ttf}[Path=fonts/,AutoFakeBold=2,AutoFakeSlant=.2]
         }{
           \IfFontExistsTF{AR PL KaitiM Big5}{
             \setCJKmainfont{AR PL KaitiM Big5}[AutoFakeBold=2,AutoFakeSlant=.2]
             \setCJKsansfont{AR PL KaitiM Big5}[AutoFakeBold=2,AutoFakeSlant=.2]
           }{\PackageError{layout_settings}{Traditional Chinese font missing}{
             Upload the supplied fonts folder, or compile with pdfLaTeX.}}}}}
  \else
    \PackageError{layout_settings}{Unsupported engine}{Use pdfLaTeX or XeLaTeX.}
  \fi
\fi
\ifPDFTeX
  \AtBeginDocument{\begin{CJK*}{UTF8}{bkai}}
  \AtEndDocument{\clearpage\end{CJK*}}
\fi

\usepackage{amsmath}
\usepackage{amsfonts}
\usepackage{amssymb}
\usepackage{amsthm}
\usepackage{bbm}
\usepackage{esint}
\usepackage{mathrsfs}
\providecommand{\widebar}[1]{\overline{#1}}

\usepackage[
    a4paper,
    left=3.18cm,
    right=3.18cm,
    top=2.54cm,
    bottom=2.54cm
]{geometry}
\usepackage{setspace}
\usepackage{enumitem}
\usepackage{comment}
\usepackage{lineno}

\usepackage{xcolor}
\usepackage{graphicx}
\ifPDFTeX\usepackage{epstopdf}\fi
\usepackage{overpic}
\usepackage{subcaption}
\usepackage{tikz}

\usepackage{booktabs}
\usepackage{multirow}
\usepackage{rotating}
\usepackage{algorithm}
\usepackage{algpseudocode}

\ifTemplateManuscript
\usepackage{titling}

\usepackage{authblk}
\usepackage[nottoc,numbib]{tocbibind}
\fi
\usepackage{titlesec}
\usepackage{array,tabularx,float,fancyhdr,lastpage}
\ifTemplateChinese\usepackage{microtype}\fi

\usepackage{hyperref}
\usepackage[nameinlink]{cleveref}

\definecolor{trueblue}{rgb}{0.0, 0.45, 0.81}
\definecolor{lightgoldenrodyellow}{rgb}{0.98, 0.98, 0.82}
\definecolor{uclablue}{rgb}{0.33, 0.41, 0.58}
\definecolor{darkcerulean}{rgb}{0.03, 0.27, 0.49}
\definecolor{carnelian}{rgb}{0.7, 0.11, 0.11}
\definecolor{lightgray}{rgb}{.25,.25,.25}
\definecolor{midgray}{rgb}{.6,.6,.6}
\definecolor{bleudefrance}{rgb}{0.19, 0.55, 0.91}
\definecolor{columbiablue}{rgb}{0.61, 0.87, 1.0}
\definecolor{oxfordblue}{rgb}{0.0, 0.13, 0.28}
\definecolor{yaleblue}{rgb}{0.06, 0.3, 0.57}
\definecolor{harvardcrimson}{rgb}{0.79, 0.0, 0.09}
\definecolor{burgundy}{rgb}{0.50, 0.0, 0.13}
\definecolor{forestgreen}{rgb}{0.0, 0.35, 0.18}
\definecolor{darkelectricblue}{rgb}{0.33, 0.41, 0.47}
\definecolor{babyblueeyes}{rgb}{0.63, 0.79, 0.95}
\definecolor{ceruleanblue}{rgb}{0.16, 0.32, 0.75}
\definecolor{sphereblue}{RGB}{129,159,247}
\definecolor{bananayellow}{rgb}{1.0, 0.88, 0.21}
\definecolor{MSBlue}{rgb}{.204,.353,.541}

\numberwithin{equation}{section}
\numberwithin{figure}{section}
\theoremstyle{plain}
\newtheorem{theorem}{\sffamily Theorem}
\newtheorem{corollary}[theorem]{\sffamily Corollary}
\newtheorem{lemma}[theorem]{\sffamily Lemma}
\newtheorem{assumption}[theorem]{\sffamily Assumption}

\setlist[enumerate]{leftmargin=.5in}
\setlist[itemize]{leftmargin=.5in}

\graphicspath{{images/}{./}}
\usetikzlibrary{
    arrows.meta,
    calc,
    positioning,
    angles,
    quotes,
    intersections,
    3d,
    perspective,
    decorations.pathreplacing,
    decorations.markings
}

\newlength{\Oldarrayrulewidth}

\hypersetup{
    colorlinks=true,
    allcolors=trueblue,
    hypertexnames=false
}
\crefname{section}{section}{sections}
\crefname{subsection}{subsection}{subsections}
\Crefname{section}{Section}{Sections}
\Crefname{subsection}{Subsection}{Subsections}
\Crefname{figure}{Figure}{Figures}
\Crefname{corollary}{Corollary}{Corollary}

\titleformat*{\section}{\Large\bfseries\sffamily\color{black}}
\titleformat*{\subsection}{\large\bfseries\sffamily\color{black}}
\titleformat*{\subsubsection}{\bfseries\sffamily\color{black}}

\makeatletter
\def\@fnsymbol#1{\ensuremath{\ifcase#1\or \dagger\or \ddagger\or
   \mathsection\or \mathparagraph\or \|\or **\or \dagger\dagger
   \or \ddagger\ddagger \else\@ctrerr\fi}}
\makeatother
\ifTemplateManuscript
\makeatletter
\posttitle{\par\end{center}\vskip 0.2em}
\renewcommand{\maketitlehookc}{\ifx\@date\@empty\predate{}\postdate{}\fi}

\makeatother
\fi

\newcommand{\argmin}{\operatornamewithlimits{argmin}}

\ifdefstring{\TemplateLegacyMath}{true}{
\def\0{\mathbf{0}}
\def\1{\mathbf{1}}

\def\c{\mathbf{c}}

\def\z{\mathbf{z}}

\def\R{\mathbb{R}}
\def\bS{\mathbb{S}}

\def\A{\mathcal{A}}
\def\E{\mathcal{E}}

\def\M{\mathcal{M}}
\def\N{\mathcal{N}}

\def\d{\mathrm{d}}

\def\O{\mathrm{O}}

\def\rS{\mathrm{S}}

\def\Mh{\mathcal{M}_h}
\def\Nh{\mathcal{N}_h}

\ifTemplateProject\fi
}{}
\definecolor{heading}{RGB}{42,68,86}
\definecolor{placeholdergray}{RGB}{90,90,90}

\ifTemplateProject
  \setlist[itemize]{leftmargin=2.2em,itemsep=0.25em,topsep=0.35em}
  \setlist[enumerate]{leftmargin=2.4em,itemsep=0.25em,topsep=0.35em}
  \titleformat{\section}{\Large\bfseries\sffamily}{\thesection.}{0.7em}{}
  \titlespacing*{\section}{0pt}{1.2em}{0.65em}
  \titleformat{\subsection}{\large\bfseries\sffamily}{\thesubsection}{0.65em}{}
  \titlespacing*{\subsection}{0pt}{1em}{0.45em}
  \titleformat{\subsubsection}{\normalsize\bfseries\sffamily}{\thesubsubsection}{0.6em}{}
  \titlespacing*{\subsubsection}{0pt}{0.8em}{0.3em}
  \hypersetup{allcolors=trueblue}
\fi
\ifTemplateChinese
  \setlist[itemize]{leftmargin=2.2em,itemsep=0.2em,topsep=0.25em}
  \titleformat{\section}{\large\bfseries\color{heading}}{\thesection.}{0.55em}{}
  \titlespacing*{\section}{0pt}{1.05em}{0.35em}
  \hypersetup{allcolors=trueblue}

\fi

\title{\LARGE\bfseries\sffamily
Relaxation and $\Gamma$-Convergence of Stretch Energy on Surface Homeomorphisms
	}
\author{\sffamily Shu-Yung Liu\thanks{Department of Mathematics, National Taiwan Normal University, Taipei, Taiwan.
  Email: \href{mailto:lii227857@gmail.com}{lii227857@gmail.com}.}}
\date{}
\hypersetup{pdftitle={Relaxation and Gamma-Convergence of Stretch Energy on Surface Homeomorphisms},pdfauthor={Shu-Yung Liu}}

\begin{document}
\maketitle

\begin{abstract}
	Area-preserving parameterization of triangular meshes is an essential task in geometry processing, yet an exactly area-preserving simplicial mapping may not exist in general.
	In this paper, we study area preservation of simplicial mappings in the limit of mesh refinement by establishing the $\Gamma$-convergence of the discrete stretch energy, a functional that measures their area distortion.
	We first introduce its limit functional, the relaxed stretch energy on surface homeomorphisms, defined through the pushforward of the area measure, and show that it is the relaxation of the stretch energy of diffeomorphisms with respect to uniform convergence.
	Then, to compare simplicial mappings with homeomorphisms, we construct a geometric lift that requires no shape-regularity assumption on the image mesh.
	Finally, we establish the $\Gamma$-convergence and show that the minimal discrete energies converge to the continuum minimum.
	Moreover, any homeomorphism arising as a limit of lifted discrete local minimizers in fixed-radius balls is area-preserving.


\end{abstract}

\section{Introduction}\label{sec:intro}
Area-preserving geometry is well studied, dating back to Moser's work~\cite{Mose65}, which shows that for two volume forms with equal total volume on a compact, connected, oriented manifold without boundary, there exists an orientation-preserving diffeomorphism pulling back one volume form to the other. It is then extended to compact oriented manifolds with boundary by Banyaga~\cite{Bany1974}, where the diffeomorphism may be chosen so that its restriction to the boundary is the identity, and to manifolds with corners by Bruveris et al.~\cite{Bruv18}. Therefore, by pulling back the volume form of the other manifold by a diffeomorphism, an area-preserving diffeomorphism between two manifolds is given by the composition of two diffeomorphisms.

In computer graphics, surface parameterization is a fundamental problem that facilitates downstream tasks, and area preservation is an important property in parameterization \cite{FlHo05, ShPR06}. For example, it is used in surface registration~\cite{LiYu24}, geometry images~\cite{GuGH02}, and computing spherical harmonic coefficients for shape analysis~\cite{BrGK95, GeSJ01, StLP04, StOX06, GeCC09}.

Several methods then use Moser's work in the computation of area-preserving parameterization. Zou et al.~\cite{ZoHG11} construct Moser's flow by Lie advection of differential forms, with a velocity field obtained from a Poisson problem. Dominitz and Tannenbaum~\cite{DoTa10} apply it to spherical domains. Moreover, Lee and Kazhdan~\cite{LeKa19} use the negative gradient of the logarithmic scale factor for spherical parameterizations. In contrast, rather than using a differential flow on the surface, Choi and Rycroft~\cite{ChRy18} apply density-equalizing maps to evolve an area ratio by diffusion on the parameter domain. It has been extended to multiply connected planar domains~\cite{LyCh24}, the hemisphere~\cite{GiCK21}, hemispheroids~\cite{ChSh25}, the torus~\cite{YaCh26}, and the sphere~\cite{LyLC24}.

However, a surface in practice is approximated by a triangular mesh, and a diffeomorphism is approximated by a simplicial mapping, which is piecewise affine. Since the Jacobian of a piecewise-linear simplicial map is constant on each triangle, the flow using differential operators requires a smoothed representation of the triangle-wise area ratio before computing the velocity field. The reality is that the error between the smoothed representation and the true triangle-wise area ratio may be uncontrolled, which may further cause instability of the algorithms.

Instead, many methods consider the variational framework on the discrete space. Gu et al.~\cite{GuLS16} focus on convex polyhedral meshes by establishing a variational principle of discrete optimal transport and the discrete Monge--Amp\`ere equation, which is used for the parameterization of disks~\cite{ZhSG13, SuCQ16} and spheres~\cite{NaSZ17, CuQW19}. Other methods work directly on triangular meshes, including intrinsic parameterization~\cite{DeMA02}, $L^2$-stretch minimization~\cite{YoBS04}, squared triangle-area distortion~\cite{KeXY22}, stretch energy minimization~\cite{YuLW19, YuLL19, YuLL20, Yueh23, SuYu24}, and authalic energy minimization~\cite{LiYu24, LiYu25, LiYu26}.

While the variational framework is effective in practice, it lacks a connection to the theory of smooth surfaces. More precisely, although the triangle-wise area distortion decreases well in computation, it remains unknown whether it produces discrete area-preserving mappings and corresponds to area-preserving homeomorphisms on surfaces.

The prior work~\cite{LiKY26} addresses this problem by completing the stretch energy minimization \cite{Yueh23}. The stretch energy, originally defined on simplicial mappings, was extended to diffeomorphisms between Riemannian surfaces, and it was shown that every critical point is area-preserving. They have also shown that the discrete stretch energy converges to its smooth counterpart by interpolating the diffeomorphism, and that the error of the optimal value is bounded by the squared mesh size. This shows that, while an exactly area-preserving simplicial mapping may not exist, the discrete problem is still meaningful.

Nonetheless, several questions remain. Because simplicial mappings and diffeomorphisms lie in different spaces, the mappings cannot be compared directly, and hence only the convergence of functional values is established. Even if we transform a simplicial mapping into a diffeomorphism, the smoothness may be lost in the limit. This motivates us to propose the relaxation of the smooth stretch energy to homeomorphisms and to establish $\Gamma$-convergence, which justifies the convergence of the variational problem.

\subsection{Main results}\label{sec:main_results}
Let $\M, \N \subset \R^3$ be two compact, connected, oriented smooth surfaces and let $f:\M \to \N$ be an orientation-preserving diffeomorphism. The prior work~\cite{LiKY26} measures the distortion of the local area ratio, defined by $J_f = f^*\d A_\N / \d A_\M$ under the area normalization
\[
|\M| := \int_\M 1 \, \d A_\M = \int_\M J_f \, \d A_\M = \int_\N 1 \, \d A_\N =: |\N|,
\]
by the \emph{stretch energy functional}
\[
\mathcal{E}_\rS(f) = \int_\M J^2_f \, \d A_\M.
\]
It is in line with the squared $L^2$ difference
\begin{equation} \label{eq:energy_gap_smooth}
\mathcal{E}_\rS(f) - |\M| = \int_\M J_f^2 \,\d A_\M -2 \int_\M J_f \, \d A_\M + \int_\M 1 \, \d A_\M = \| J_f - 1\|_{L^2(\M)}^2 \geq 0,
\end{equation}
where the equality holds if and only if $f$ is area-preserving.
Let $\Mh$ be a triangulated surface approximating $\M$ with maximal triangle diameter $h$. The discrete counterpart for a simplicial map $f_h:\Mh \to \Nh$ is
\[
E_\rS(f_h)= \int_{\Mh} J_{f_h}^2 \, \d A  = \sum_{\tau \in \mathrm{triangles}} \int_\tau J_{f_h}|_\tau^2 \, \d A.
\]

In this work, we focus on the mappings themselves. Since diffeomorphisms may converge to a homeomorphism that is not differentiable, we consider the space of orientation-preserving homeomorphisms with the uniform distance
\[
\mathrm{Homeo}^+(\M,\N),\qquad
d_H(f, g) = \sup_{x \in \M} |f(x) - g(x) |.
\]
Defining the measure $\mu_f(B)=|f(B)|_\N$ on $\M$, we propose the \emph{relaxed stretch energy}
\begin{equation*}
	\widebar{\E_\rS}(f) =
	\begin{cases}
		\displaystyle\int_\M \widehat J_f^2 \, \d A_\M
		& \text{if }\mu_f = \widehat J_f \,\d A_\M\text{ for some } \widehat J_f\in L^2(\M),\\
		+\infty & \text{otherwise},
	\end{cases}
\end{equation*}
and show that it is exactly the relaxation of $\E_\rS$ with respect to $d_H$; see Corollary~\ref{cor:energy_relax}.
To compare discrete maps with homeomorphisms, we construct a lift $\mathrm{L}_h$ that sends each admissible discrete map $f_h$, without shape regularity on its image, to $\mathrm{L}_hf_h\in\mathrm{Homeo}^+(\M,\N)$. Our main result is the $\Gamma$-convergence \cite{DaMa93, Brai02}
\[
\widebar{E_\rS}_h \xrightarrow[\;h\to0\;]{\;\Gamma(d_H) \;} \widebar{\E_\rS},
\qquad
\widebar{E_\rS}_h(u) = \inf \big\{ E_\rS(f_h) \;\big|\; f_h\text{ admissible},\ \mathrm{L}_h f_h = u \big\};
\]
see Theorem~\ref{thm:main}. As consequences, the minimal discrete energies converge to $\min\widebar{\E_\rS}=|\M|$ (Corollary~\ref{cor:min_converge}), and local minimality passes to the limit: if each $f_h$ minimizes $E_\rS$ among admissible maps whose lifts lie within a fixed distance $\rho$ of $\mathrm{L}_hf_h$, then every limit of $\mathrm{L}_hf_h$ in $\mathrm{Homeo}^+(\M,\N)$ minimizes $\widebar{\E_\rS}$ on the ball of radius $\rho$, and is therefore area-preserving (Theorem~\ref{thm:local_min_limit}).

\subsection{Outline of the proof}\label{sec:proof_idea}

The $\Gamma$-convergence includes both the liminf and limsup inequalities. We prove the liminf inequality by the lower semicontinuity and the energy comparison for the lift:
\begin{equation} \label{eq:idea_liminf}
\widebar{\E_\rS}(u)
\leq \liminf_{h\to0}\widebar{\E_\rS}(\mathrm{L}_hf_h)
\leq \liminf_{h\to0}E_\rS(f_h),
\end{equation}
which are introduced in Section~\ref{sec:lsc} and Section~\ref{sec:lift}, respectively.
Here, we require that the projection from the image mesh $f_h(\Mh)$ to the smooth manifold is bijective without shape regularity, and that the lift of $f_h$, $\mathrm{L}_hf_h$, is in $\mathrm{Homeo}^+(\M,\N)$.

For the limsup inequality, given $u \in \mathrm{Homeo}^+(\M,\N)$, we show it by two approximations:
\begin{equation} \label{eq:idea_limsup}
E_\rS(\mathrm{I}_hg_j)
\xrightarrow[h\to0]{\text{interpolation}}
\E_\rS(g_j)
\xrightarrow[j\to\infty]{\text{smoothing}}
\widebar{\E_\rS}(u),
\end{equation}
where the first one indicates that the piecewise-linear interpolant approaches a fixed diffeomorphism, shown in Section~\ref{sec:interp}, and the second one represents that the diffeomorphisms approach a homeomorphism, shown in Section~\ref{sec:density}.

With these two pieces, we establish the main theorem, derive infimum convergence without compactness, and obtain the limit of local minimizers in Section~\ref{sec:main}.

\section{Lower semicontinuity}\label{sec:lsc}
As mentioned in Section~\ref{sec:main_results}, the prior work \cite{LiKY26} focuses on the stretch energy $\E_\rS(f)$ defined on orientation-preserving diffeomorphisms. However, the smoothness may be lost during the limit, and thus assigning $+\infty$ to all other orientation-preserving homeomorphisms would not be lower semicontinuous under uniform convergence.

For instance, under longitude--height coordinates $(\theta,z)$ on $\bS^2$ away from the poles, consider the sequence $\{f_j\}$ (Figure~\ref{fig:lsc_example}):
\[
f_j(\theta, z) = \big(\theta + \sqrt{z^2 + j^{-2}}, z\big) \quad \longrightarrow \quad
f(\theta,z) = (\theta + |z|, z).
\]
Since $0 \leq \sqrt{z^2+j^{-2}}-|z| \leq j^{-1}$, both $f_j$ and $f_j^{-1}$ converge uniformly to $f$ and $f^{-1}$, respectively. In longitude--height coordinates, the sphere is parametrized by
\[
X(\theta,z)=\big(\sqrt{1-z^2}\cos\theta,\sqrt{1-z^2}\sin\theta,z\big),
\qquad -1<z<1.
\]
Since $|\partial_\theta X\times\partial_z X|=1$, its area element is $\d A_{\bS^2}=\d\theta\,\d z$.
The coordinate Jacobian matrix of $f_j$ is
\[
D_{(\theta,z)}f_j=
\begin{pmatrix}
1 & \dfrac{z}{\sqrt{z^2+j^{-2}}} \\
0 & 1
\end{pmatrix},
\qquad \det D_{(\theta,z)}f_j=1.
\]
Thus each $f_j$ preserves spherical area.
The limit map $f$ also preserves area, since for each fixed $z$ it translates $\theta$ modulo $2\pi$, preserving $\d\theta$, and the poles have zero area.
Nonetheless, the limit $f$ is non-smooth and therefore
\[
\E_\rS(f) = +\infty > \liminf_{j \to \infty} \E_\rS(f_j) = 4 \pi,
\]
which violates lower semicontinuity.

\begin{figure}[tb]
\centering
\begin{tikzpicture}[
  >=Latex,
  font=\footnotesize,
  every node/.style={inner sep=1.5pt},
  eqv/.style={gray!60, line width=0.4pt},
  eqh/.style={gray!50, line width=0.35pt, densely dashed},
  mer/.style={gray!55, line width=0.5pt, densely dotted},
  merh/.style={draw=none},
  fj/.style={gray!55, line width=0.9pt},
  fjh/.style={gray!40, line width=0.5pt, densely dashed},
  fl/.style={black, line width=0.9pt},
  flh/.style={black!60, line width=0.5pt, densely dashed},
  axis/.style={->, line width=0.5pt},
  ttl/.style={font=\footnotesize, align=center}
]
\begin{scope}
  \fill[gray!6] (0,0) circle (2);
  \draw[line width=0.6pt] (0,0) circle (2);
  \draw[eqv] (0.000,-0.618) -- (0.105,-0.617) -- (0.209,-0.615) -- (0.313,-0.610) -- (0.416,-0.605) -- (0.518,-0.597) -- (0.618,-0.588) -- (0.717,-0.577) -- (0.813,-0.565) -- (0.908,-0.551) -- (1.000,-0.535) -- (1.089,-0.518) -- (1.176,-0.500) -- (1.259,-0.480) -- (1.338,-0.459) -- (1.414,-0.437) -- (1.486,-0.414) -- (1.554,-0.389) -- (1.618,-0.363) -- (1.677,-0.337) -- (1.732,-0.309) -- (1.782,-0.281) -- (1.827,-0.251) -- (1.867,-0.221) -- (1.902,-0.191) -- (1.932,-0.160) -- (1.956,-0.128) -- (1.975,-0.097) -- (1.989,-0.065) -- (1.997,-0.032) -- (2.000,-0.000);
  \draw[eqh] (2.000,-0.000) -- (1.997,0.032) -- (1.989,0.065) -- (1.975,0.097) -- (1.956,0.128) -- (1.932,0.160) -- (1.902,0.191) -- (1.867,0.221) -- (1.827,0.251) -- (1.782,0.281) -- (1.732,0.309) -- (1.677,0.337) -- (1.618,0.363) -- (1.554,0.389) -- (1.486,0.414) -- (1.414,0.437) -- (1.338,0.459) -- (1.259,0.480) -- (1.176,0.500) -- (1.089,0.518) -- (1.000,0.535) -- (0.908,0.551) -- (0.813,0.565) -- (0.717,0.577) -- (0.618,0.588) -- (0.518,0.597) -- (0.416,0.605) -- (0.313,0.610) -- (0.209,0.615) -- (0.105,0.617) -- (0.000,0.618) -- (-0.105,0.617) -- (-0.209,0.615) -- (-0.313,0.610) -- (-0.416,0.605) -- (-0.518,0.597) -- (-0.618,0.588) -- (-0.717,0.577) -- (-0.813,0.565) -- (-0.908,0.551) -- (-1.000,0.535) -- (-1.089,0.518) -- (-1.176,0.500) -- (-1.259,0.480) -- (-1.338,0.459) -- (-1.414,0.437) -- (-1.486,0.414) -- (-1.554,0.389) -- (-1.618,0.363) -- (-1.677,0.337) -- (-1.732,0.309) -- (-1.782,0.281) -- (-1.827,0.251) -- (-1.867,0.221) -- (-1.902,0.191) -- (-1.932,0.160) -- (-1.956,0.128) -- (-1.975,0.097) -- (-1.989,0.065) -- (-1.997,0.032) -- (-2.000,0.000);
  \draw[eqv] (-2.000,0.000) -- (-1.997,-0.032) -- (-1.989,-0.065) -- (-1.975,-0.097) -- (-1.956,-0.128) -- (-1.932,-0.160) -- (-1.902,-0.191) -- (-1.867,-0.221) -- (-1.827,-0.251) -- (-1.782,-0.281) -- (-1.732,-0.309) -- (-1.677,-0.337) -- (-1.618,-0.363) -- (-1.554,-0.389) -- (-1.486,-0.414) -- (-1.414,-0.437) -- (-1.338,-0.459) -- (-1.259,-0.480) -- (-1.176,-0.500) -- (-1.089,-0.518) -- (-1.000,-0.535) -- (-0.908,-0.551) -- (-0.813,-0.565) -- (-0.717,-0.577) -- (-0.618,-0.588) -- (-0.518,-0.597) -- (-0.416,-0.605) -- (-0.313,-0.610) -- (-0.209,-0.615) -- (-0.105,-0.617) -- (-0.000,-0.618);
  \draw[merh] (-0.050,-1.923) -- (-0.211,-1.964) -- (-0.293,-1.969) -- (-0.356,-1.966) -- (-0.408,-1.958);
  \draw[mer] (-0.408,-1.958) -- (-0.454,-1.947) -- (-0.494,-1.933) -- (-0.531,-1.918) -- (-0.565,-1.902) -- (-0.596,-1.885) -- (-0.626,-1.866) -- (-0.653,-1.847) -- (-0.679,-1.827) -- (-0.703,-1.806) -- (-0.726,-1.785) -- (-0.748,-1.763) -- (-0.769,-1.741) -- (-0.788,-1.718) -- (-0.807,-1.695) -- (-0.825,-1.671) -- (-0.842,-1.647) -- (-0.859,-1.623) -- (-0.875,-1.598) -- (-0.890,-1.574) -- (-0.904,-1.548) -- (-0.918,-1.523) -- (-0.931,-1.497) -- (-0.944,-1.471) -- (-0.956,-1.445) -- (-0.967,-1.419) -- (-0.978,-1.392) -- (-0.989,-1.365) -- (-0.999,-1.338) -- (-1.009,-1.311) -- (-1.018,-1.283) -- (-1.027,-1.256) -- (-1.035,-1.228) -- (-1.043,-1.200) -- (-1.051,-1.171) -- (-1.058,-1.143) -- (-1.065,-1.114) -- (-1.071,-1.086) -- (-1.077,-1.057) -- (-1.083,-1.028) -- (-1.088,-0.998) -- (-1.093,-0.969) -- (-1.098,-0.939) -- (-1.103,-0.910) -- (-1.107,-0.880) -- (-1.110,-0.850) -- (-1.114,-0.820) -- (-1.117,-0.789) -- (-1.119,-0.759) -- (-1.122,-0.728) -- (-1.124,-0.698) -- (-1.125,-0.667) -- (-1.127,-0.636) -- (-1.128,-0.604) -- (-1.129,-0.573) -- (-1.129,-0.542) -- (-1.129,-0.510) -- (-1.129,-0.478) -- (-1.129,-0.446) -- (-1.128,-0.414) -- (-1.127,-0.382) -- (-1.125,-0.350) -- (-1.124,-0.318) -- (-1.122,-0.285) -- (-1.119,-0.252) -- (-1.117,-0.219) -- (-1.114,-0.186) -- (-1.110,-0.153) -- (-1.107,-0.120) -- (-1.103,-0.086) -- (-1.098,-0.053) -- (-1.093,-0.019) -- (-1.088,0.015) -- (-1.083,0.049) -- (-1.077,0.083) -- (-1.071,0.118) -- (-1.065,0.152) -- (-1.058,0.187) -- (-1.051,0.222) -- (-1.043,0.257) -- (-1.035,0.292) -- (-1.027,0.328) -- (-1.018,0.364) -- (-1.009,0.399) -- (-0.999,0.436) -- (-0.989,0.472) -- (-0.978,0.508) -- (-0.967,0.545) -- (-0.956,0.582) -- (-0.944,0.619) -- (-0.931,0.656) -- (-0.918,0.694) -- (-0.904,0.732) -- (-0.890,0.770) -- (-0.875,0.808) -- (-0.859,0.847) -- (-0.842,0.886) -- (-0.825,0.926) -- (-0.807,0.966) -- (-0.788,1.006) -- (-0.769,1.046) -- (-0.748,1.087) -- (-0.726,1.129) -- (-0.703,1.171) -- (-0.679,1.214) -- (-0.653,1.257) -- (-0.626,1.301) -- (-0.596,1.346) -- (-0.565,1.392) -- (-0.531,1.439) -- (-0.494,1.487) -- (-0.454,1.537) -- (-0.408,1.589) -- (-0.356,1.644) -- (-0.293,1.704) -- (-0.211,1.773) -- (-0.050,1.877);
  \draw[fjh] (0.039,-1.925) -- (0.158,-1.973) -- (0.213,-1.983) -- (0.249,-1.984);
  \draw[fj] (0.249,-1.984) -- (0.275,-1.980) -- (0.294,-1.973) -- (0.308,-1.963) -- (0.317,-1.952) -- (0.322,-1.940) -- (0.324,-1.926) -- (0.324,-1.911) -- (0.321,-1.895) -- (0.316,-1.879) -- (0.309,-1.861) -- (0.300,-1.843) -- (0.289,-1.825) -- (0.277,-1.805) -- (0.263,-1.786) -- (0.248,-1.765) -- (0.232,-1.744) -- (0.215,-1.723) -- (0.197,-1.701) -- (0.178,-1.679) -- (0.158,-1.656) -- (0.137,-1.633) -- (0.116,-1.609) -- (0.094,-1.585) -- (0.071,-1.561) -- (0.048,-1.536) -- (0.025,-1.511) -- (0.001,-1.486) -- (-0.023,-1.460) -- (-0.048,-1.433) -- (-0.072,-1.407) -- (-0.097,-1.380) -- (-0.121,-1.352) -- (-0.146,-1.325) -- (-0.170,-1.297) -- (-0.195,-1.269) -- (-0.219,-1.240) -- (-0.242,-1.211) -- (-0.266,-1.182) -- (-0.288,-1.153) -- (-0.311,-1.123) -- (-0.332,-1.094) -- (-0.353,-1.063) -- (-0.374,-1.033) -- (-0.393,-1.003) -- (-0.411,-0.972) -- (-0.428,-0.941) -- (-0.445,-0.910) -- (-0.460,-0.879) -- (-0.473,-0.848) -- (-0.485,-0.817) -- (-0.496,-0.786) -- (-0.505,-0.754) -- (-0.513,-0.723) -- (-0.519,-0.691) -- (-0.524,-0.659) -- (-0.526,-0.628) -- (-0.527,-0.596) -- (-0.526,-0.565) -- (-0.524,-0.533) -- (-0.519,-0.501) -- (-0.513,-0.469) -- (-0.505,-0.437) -- (-0.496,-0.405) -- (-0.485,-0.374) -- (-0.473,-0.341) -- (-0.460,-0.309) -- (-0.445,-0.277) -- (-0.428,-0.245) -- (-0.411,-0.212) -- (-0.393,-0.179) -- (-0.374,-0.146) -- (-0.353,-0.113) -- (-0.332,-0.080) -- (-0.311,-0.047) -- (-0.288,-0.013) -- (-0.266,0.021) -- (-0.242,0.055) -- (-0.219,0.090) -- (-0.195,0.125) -- (-0.170,0.160) -- (-0.146,0.195) -- (-0.121,0.231) -- (-0.097,0.267) -- (-0.072,0.303) -- (-0.048,0.340) -- (-0.023,0.377) -- (0.001,0.415) -- (0.025,0.452) -- (0.048,0.491) -- (0.071,0.529) -- (0.094,0.568) -- (0.116,0.607) -- (0.137,0.647) -- (0.158,0.687) -- (0.178,0.728) -- (0.197,0.769) -- (0.215,0.811) -- (0.232,0.853) -- (0.248,0.895) -- (0.263,0.938) -- (0.277,0.982) -- (0.289,1.026) -- (0.300,1.070) -- (0.309,1.116) -- (0.316,1.162) -- (0.321,1.209) -- (0.324,1.256) -- (0.324,1.305) -- (0.322,1.354) -- (0.317,1.405) -- (0.308,1.457) -- (0.294,1.511) -- (0.275,1.567) -- (0.249,1.626) -- (0.213,1.690) -- (0.158,1.764) -- (0.039,1.875);
  \draw[flh] (0.035,-1.926) -- (0.140,-1.976) -- (0.186,-1.987) -- (0.216,-1.988);
  \draw[fl] (0.216,-1.988) -- (0.236,-1.985) -- (0.250,-1.978) -- (0.258,-1.969) -- (0.262,-1.958) -- (0.263,-1.945) -- (0.260,-1.932) -- (0.255,-1.917) -- (0.248,-1.901) -- (0.238,-1.884) -- (0.226,-1.867) -- (0.212,-1.849) -- (0.197,-1.830) -- (0.180,-1.810) -- (0.162,-1.790) -- (0.142,-1.770) -- (0.121,-1.749) -- (0.098,-1.727) -- (0.075,-1.705) -- (0.051,-1.682) -- (0.025,-1.659) -- (-0.001,-1.635) -- (-0.028,-1.611) -- (-0.056,-1.586) -- (-0.084,-1.561) -- (-0.114,-1.535) -- (-0.143,-1.509) -- (-0.174,-1.483) -- (-0.205,-1.456) -- (-0.236,-1.429) -- (-0.268,-1.401) -- (-0.300,-1.373) -- (-0.332,-1.344) -- (-0.365,-1.315) -- (-0.398,-1.286) -- (-0.431,-1.256) -- (-0.464,-1.226) -- (-0.498,-1.196) -- (-0.531,-1.165) -- (-0.564,-1.133) -- (-0.598,-1.102) -- (-0.631,-1.070) -- (-0.665,-1.037) -- (-0.698,-1.004) -- (-0.731,-0.971) -- (-0.763,-0.938) -- (-0.796,-0.904) -- (-0.828,-0.870) -- (-0.860,-0.835) -- (-0.892,-0.800) -- (-0.923,-0.765) -- (-0.954,-0.730) -- (-0.985,-0.694) -- (-1.015,-0.658) -- (-1.044,-0.621) -- (-1.073,-0.584) -- (-1.101,-0.547) -- (-1.129,-0.510) -- (-1.101,-0.484) -- (-1.073,-0.458) -- (-1.044,-0.431) -- (-1.015,-0.404) -- (-0.985,-0.377) -- (-0.954,-0.350) -- (-0.923,-0.322) -- (-0.892,-0.294) -- (-0.860,-0.265) -- (-0.828,-0.236) -- (-0.796,-0.207) -- (-0.763,-0.178) -- (-0.731,-0.148) -- (-0.698,-0.118) -- (-0.665,-0.087) -- (-0.631,-0.056) -- (-0.598,-0.025) -- (-0.564,0.007) -- (-0.531,0.039) -- (-0.498,0.071) -- (-0.464,0.104) -- (-0.431,0.137) -- (-0.398,0.171) -- (-0.365,0.205) -- (-0.332,0.239) -- (-0.300,0.274) -- (-0.268,0.309) -- (-0.236,0.345) -- (-0.205,0.381) -- (-0.174,0.417) -- (-0.143,0.454) -- (-0.114,0.492) -- (-0.084,0.529) -- (-0.056,0.568) -- (-0.028,0.606) -- (-0.001,0.645) -- (0.025,0.685) -- (0.051,0.725) -- (0.075,0.766) -- (0.098,0.807) -- (0.121,0.848) -- (0.142,0.891) -- (0.162,0.933) -- (0.180,0.977) -- (0.197,1.020) -- (0.212,1.065) -- (0.226,1.110) -- (0.238,1.156) -- (0.248,1.203) -- (0.255,1.250) -- (0.260,1.299) -- (0.263,1.348) -- (0.262,1.399) -- (0.258,1.452) -- (0.250,1.506) -- (0.236,1.562) -- (0.216,1.622) -- (0.186,1.687) -- (0.140,1.761) -- (0.035,1.875);
  \coordinate (corner) at (-1.129,-0.510);
  \coordinate (cj) at (-0.527,-0.596);
  \coordinate (lf) at (0.031,0.694);
  \coordinate (lfj) at (-0.071,0.304);
  \coordinate (lm) at (-0.747,1.089);
  \fill (corner) circle (1.4pt);
  \draw[line width=0.35pt] (corner) -- ++(-0.55,-0.45) node[anchor=north east, inner sep=1pt] {corner at $z=0$};
  \node[anchor=west, text=black] at ($(lf)+(0.08,0.02)$) {$f$};
  \node[anchor=west, text=gray!50!black] at ($(cj)+(0.06,0.03)$) {$f_j$};
  \node[anchor=east, text=gray!50!black] at ($(lm)+(-0.05,0.05)$) {meridian};
  \node[anchor=west, text=gray!50!black] at (2.02,-0.30) {$z=0$};
  \node[ttl] at (0,-2.55) {(a) Image of a meridian on $\bS^2$};
\end{scope}
\begin{scope}[xshift=5.9cm, x=0.62cm, y=1.7cm]
  \fill[gray!6] (-3.14,-1) rectangle (3.14,1);
  \draw[gray!60, line width=0.4pt] (-3.14,-1) rectangle (3.14,1);
  \draw[gray!50, densely dashed, line width=0.4pt] (-3.14,0) -- (3.14,0) node[anchor=west, text=gray!50!black] {$z=0$};
  \foreach \c in {-2.4,-0.6,1.2}{
    \draw[mer] (\c,-1) -- (\c,1);
    \draw[fj] plot[domain=-1:1, samples=81, variable=\z] ({\c+sqrt(\z*\z+1/9)},{\z});
    \draw[fl] plot[domain=-1:1, samples=81, variable=\z] ({\c+abs(\z)},{\z});
  }
  \fill (-0.6,0) circle (1.4pt);
  \node[anchor=west, text=black] at (2.05,0.62) {$f$};
  \node[anchor=west, text=gray!50!black, inner sep=1pt] at (1.55,0.08) {$f_j$};
  \draw[axis] (-3.4,-1) -- (3.55,-1) node[below] {$\theta$};
  \draw[axis] (-3.14,-1.12) -- (-3.14,1.2) node[left] {$z$};
  \node[anchor=east] at (-3.14,1) {$1$};
  \node[anchor=east] at (-3.14,-0.92) {$-1$};
  \node[align=left, anchor=west, font=\scriptsize] at (3.25,0.62) {$\det D f_j=1$};
  \node[ttl] at (0,-1.5) {(b) Longitude--height coordinates $(\theta,z)$};
\end{scope}
\end{tikzpicture}
\caption{Smoothness is lost in the limit: a meridian (dotted) and its images under $f_j$ (gray) and $f$ (black), (a) on $\bS^2$ and (b) in the coordinates $(\theta,z)$.}
\label{fig:lsc_example}
\end{figure}
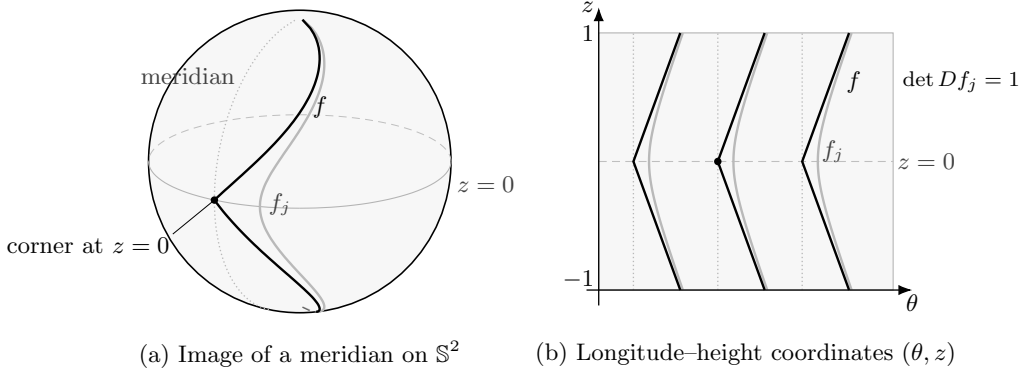

Consequently, rather than measuring the area by the pullback area form, we consider the pushforward measure
\begin{equation} \label{eq:area_measure}
	\mu_f = (f^{-1})_\# \d A_\N
\end{equation}
such that
\[
\mu_f(B)
= \int_{(f^{-1})^{-1}(B)}\d A_\N
= |f(B)|_\N, \qquad B\in\mathcal{B}(\M),
\]
where $\mathcal{B}(\M)$ is the Borel $\sigma$-algebra on $\M$.
We then extend the local area ratio $J_f = f^*\d A_\N / \d A_\M$ to $\widehat J_f$ when
\begin{equation}
\mu_f = \widehat J_f \,\d A_\M\text{ for some } \widehat J_f\in L^2(\M)
\end{equation}
and define the \emph{relaxed stretch energy functional} as
\begin{equation}\label{eq:relaxed_energy}
\widebar{\E_\rS}(f) = \int_\M \widehat J_f^2 \, \d A_\M
\end{equation}
and set $\widebar{\E_\rS}(f) = + \infty$ otherwise. If $\widebar{\E_\rS}(f) < \infty$ and
\[
|\M| = \int_\M 1 \, \d A_\M = \int_\M \widehat J_f \, \d A_\M = \mu_f(\M) = |\N|,
\]
it inherits \eqref{eq:energy_gap_smooth} immediately
\begin{equation}\label{eq:energy_gap}
\widebar{\E_\rS}(f) - |\M| = \int_\M \widehat J_f^2 \,\d A_\M -2 \int_\M \widehat J_f \, \d A_\M + \int_\M 1 \, \d A_\M = \| \widehat J_f - 1\|^2_{L^2(\M)} \geq 0,
\end{equation}
where equality holds if and only if $\mu_f=\d A_\M$, i.e., $f$ preserves area.

To show that such a functional is a proper relaxation of the original functional, we establish the lower semicontinuity in this section. We first show that the one-sided distance $d_H$ already controls the inverses.

\begin{lemma}\label{lem:inverse_uniform}
Let $f_j,f\in\mathrm{Homeo}^+(\M,\N)$. If $f_j\to f$ uniformly on $\mathcal M$, then
$f_j^{-1}\to f^{-1}$ uniformly on $\mathcal N$.
\end{lemma}

\begin{proof}
We prove it by contradiction. Suppose there exist $\varepsilon>0$, a subsequence $\{j_k\}$, and points $y_{j_k}\in\N$ such that
\[
|f_{j_k}^{-1}(y_{j_k})-f^{-1}(y_{j_k})|\geq\varepsilon
\quad\text{for every }k.
\]
Set $x_{j_k}=f_{j_k}^{-1}(y_{j_k})$. Since $\M$ is compact, we choose a subsequence $\{x_{j_{k_\ell}}\}$ converging to $x$ in $\M$. The triangle inequality gives
\begin{align*}
	|y_{j_{k_\ell}} - f(x)|
	\leq
	\underbrace{|f_{j_{k_\ell}}(x_{j_{k_\ell}}) - f(x_{j_{k_\ell}})|}_{\leq \| f_{j_{k_\ell}} - f\|_\infty \to 0} +
	\underbrace{ |f(x_{j_{k_\ell}}) - f(x)| }_{\to 0},
\end{align*}
where the second term tends to zero by the continuity of $f$. Therefore,
\begin{equation}
f_{j_{k_\ell}}(x_{j_{k_\ell}}) \to f(x) \qquad\text{as }\ell\to\infty,
\label{eq:inverse_uniform_eq1}
\end{equation}
and thus $y_{j_{k_\ell}}\to y:=f(x)$ and $x=f^{-1}(y)$. By the triangle inequality once more, we obtain
\[
|f_{j_{k_\ell}}^{-1} (y_{j_{k_\ell}}) - f^{-1} (y_{j_{k_\ell}}) | \leq
\underbrace{|f_{j_{k_\ell}}^{-1} (y_{j_{k_\ell}})  - x|}_{=|x_{j_{k_\ell}} - x| \to 0} +
\underbrace{ |x - f^{-1} (y_{j_{k_\ell}})|}_{= |f^{-1}(y) - f^{-1} (y_{j_{k_\ell}})| \to 0}
\]
where the second term tends to zero by the continuity of $f^{-1}$. This contradicts its lower bound $\varepsilon$.
\end{proof}

Since the relaxed stretch energy \eqref{eq:relaxed_energy} is defined through the pushforward, the pushforward integration gives the weak convergence of the area measures and hence the lower semicontinuity.

\begin{lemma} \label{lem:weak_conv_area}
If $f_j \to f$ in $\mathrm{Homeo}^+(\M,\N)$, then the area measures defined in \eqref{eq:area_measure} satisfy $\mu_{f_j} \rightharpoonup \mu_f$ weakly as finite measures on $\M$. Moreover, $\widebar{\E_\rS}$ is lower semicontinuous on $\mathrm{Homeo}^+(\M,\N)$; that is,
\begin{equation} \label{eq:lower_semiconti}
 \widebar{\E_\rS}(f) \leq \liminf_{j \to \infty} \widebar{\E_\rS}(f_j).
\end{equation}
\end{lemma}

\begin{proof}
For $\varphi \in C(\M)$, the pushforward integration formula applied to \eqref{eq:area_measure} gives
\[
\int_\M \varphi \,\d \mu_{f_j} = \int_\M \varphi \,\d ((f_j^{-1})_\# \d A_\N) = \int_\N \varphi \circ f_j^{-1} \,\d A_\N.
\]
By Lemma~\ref{lem:inverse_uniform}, $f_j^{-1} \to f^{-1}$ uniformly. Since $\varphi$ is continuous on the compact surface $\M$, it is uniformly continuous, and $\varphi \circ f_j^{-1} \to \varphi \circ f^{-1}$ uniformly as well.
Since $|\N|<\infty$, we obtain
\begin{equation} \label{eq:weak_conv_mu}
\int_\M \varphi\,\d\mu_{f_j}
= \int_\N \varphi \circ f_j^{-1}\,\d A_\N
\longrightarrow \int_\N \varphi \circ f^{-1}\,\d A_\N
= \int_\M \varphi\,\d\mu_f.
\end{equation}
Since this holds for arbitrary $\varphi\in C(\M)$, we obtain $\mu_{f_j}\rightharpoonup\mu_f$.

Next, the lower semicontinuity \eqref{eq:lower_semiconti} holds if $\liminf_{j\to\infty}\widebar{\E_\rS}(f_j)=+\infty$. We therefore assume the liminf is finite and choose a subsequence $\{f_{j_k}\}$ with finite energies such that
\[
\widebar{\E_\rS}(f_{j_k}) \to \liminf_{j \to \infty}\widebar{\E_\rS}(f_j).
\]
Since the energies converge to a finite limit, $\|\widehat J_{f_{j_k}}\|^2_{L^2(\M)}=\widebar{\E_\rS}(f_{j_k})$ is bounded, i.e., $\{\widehat J_{f_{j_k}}\}$ is a bounded sequence in $L^2(\M)$. Since every bounded sequence in a Hilbert space has a weakly convergent subsequence, we have
\[
\widehat J_{f_{j_{k_\ell}}} \rightharpoonup \widehat J \quad \text{weakly in }L^2(\M).
\]
In addition, every $\varphi\in C(\M)$ belongs to $L^2(\M)$ due to compactness of $\M$ and thus the weak convergence gives
\begin{equation} \label{eq:weak_conv_area_1}
\int_\M \varphi \widehat J_{f_{j_{k_\ell}}} \, \d A_\M \to \int_\M \varphi \widehat J \d A_\M.
\end{equation}
On the other hand, \eqref{eq:weak_conv_mu} gives
\begin{equation} \label{eq:weak_conv_area_2}
\int_\M \varphi \widehat J_{f_{j_{k_\ell}}} \, \d A_\M  =
\int_\M \varphi \, \d \mu_{f_{j_{k_\ell}}} \to \int_\M \varphi \d \mu_f.
\end{equation}
Taking \eqref{eq:weak_conv_area_1} and \eqref{eq:weak_conv_area_2} for every $\varphi \in C(\M)$, we obtain
\[
\mu_f = \widehat J \, \d A_\M, \qquad \widehat J = \widehat J_f\quad\text{almost everywhere}.
\]
Hence, the weak lower semicontinuity of the $L^2$ norm implies
\begin{align*}
	\widebar{\E_\rS}(f)
	= \|\widehat J_f\|^2_{L^2(\M)}
	&\leq \liminf_{\ell\to\infty}\|\widehat J_{f_{j_{k_\ell}}}\|^2_{L^2(\M)} \\
	&= \lim_{\ell\to\infty}\widebar{\E_\rS}(f_{j_{k_\ell}})
	= \liminf_{j\to\infty}\widebar{\E_\rS}(f_j),
\end{align*}
where the last equality follows from our choice of subsequence.
\end{proof}

\section{Approximation by diffeomorphisms}
\label{sec:density}
In this section, we will show the approximation of $\widebar{\E_\rS}$ by smooth $\E_\rS$, which also suggests that  $\widebar{\E_\rS}$ is indeed a relaxation of $\E_\rS$.

In dimension two, the previous works~\cite{Mois77,Munk60} have shown that a homeomorphism between compact smooth surfaces can be uniformly approximated by smooth diffeomorphisms. However, the uniform approximation alone does not ensure the Jacobian approximation, and hence the energy approximation. On the other hand, the work of Moser and Banyaga \cite{Mose65,Bany1974} gives the control of smooth area but not the displacement from the original map.

This motivates us to combine both techniques, where we approximate the homeomorphism by smooth diffeomorphisms while locally correcting the area form by a Moser flow. To do this, we introduce a cell decomposition and provide an estimate in the following lemma.

\begin{lemma} \label{lem:cell_decompose}
	For each $\delta>0$, let $\mathcal{P}_\delta(\M)$ be a finite cell decomposition of $\M$ whose maximum diameter is less than $\delta$. Then, for $\widehat J_f\in L^2(\M)$, the cellwise average, defined on each cell $Q\in\mathcal{P}_\delta(\M)$, satisfies
	\begin{equation} \label{eq:cell_decompose}
		\mathrm{C}_\delta\widehat J_f :=\frac{1}{|Q|}\int_Q \widehat J_f\,\d A_\M
		\longrightarrow  \widehat J_f
		\quad\text{in }L^2(\M)
	\end{equation}
	as $\delta \to 0$, for any choice of the decompositions $\mathcal{P}_\delta(\M)$.
\end{lemma}

\begin{proof}
	Since $C(\M)$ is dense in $L^2(\M)$, for every $\varepsilon>0$, there exists
	$v\in C(\M)$ such that $\|\widehat J_f-v\|_{L^2(\M)}<\varepsilon$.
	The triangle inequality gives
	\begin{equation} \label{eq:cell_decompose_1}
		\|\mathrm{C}_\delta\widehat J_f-\widehat J_f\|_{L^2(\M)}
		\leq \|\mathrm{C}_\delta(\widehat J_f-v)\|_{L^2(\M)}
		+\|\mathrm{C}_\delta v-v\|_{L^2(\M)}
		+\|v-\widehat J_f\|_{L^2(\M)}.
	\end{equation}
	For the second term of \eqref{eq:cell_decompose_1}, since $v$ is continuous on the compact surface $\M$, it is uniformly continuous. Therefore,
	\[
	\bigg(\int_\M |\mathrm{C}_\delta v - v|^2 \d A_\M \bigg)^{\tfrac{1}{2}}
	\leq |\M|^{\tfrac{1}{2}}  \sup_{x\in\M}  |\mathrm{C}_\delta v(x) - v(x)|
	\leq |\M|^{\tfrac{1}{2}}  \sup_{d_\M(x,y)\leq\delta}|v(y)-v(x)|,
	\]
	which tends to 0 as $\delta\to0$.

	For $u \in L^2(\M)$, the Cauchy--Schwarz inequality gives
	\[
	\|\mathrm{C}_\delta u\|_{L^2(\M)}^2
	=\sum_{Q\in\mathcal{P}_\delta(\M)}
	\frac{1}{|Q|}\left|\int_Q u\,\d A_\M\right|^2
	\leq\sum_{Q\in\mathcal{P}_\delta(\M)}\int_Q|u|^2\,\d A_\M
	=\|u\|_{L^2(\M)}^2.
	\]
	Hence, the first term of \eqref{eq:cell_decompose_1} is bounded by the third term, and \eqref{eq:cell_decompose_1} becomes
	\[
	\|\mathrm{C}_\delta\widehat J_f-\widehat J_f\|_{L^2(\M)}
	\leq 2\|\widehat J_f-v\|_{L^2(\M)}
	+|\M|^{\tfrac{1}{2}} \sup_{d_\M(x,y)\leq\delta}|v(y)-v(x)|.
	\]
	This proves the statement by letting $\delta\to0$ first and then $\varepsilon\to0$.

\end{proof}

Next, for a homeomorphism $f\in\mathrm{Homeo}^+(\M,\N)$ with $\widebar{\E_\rS}(f)<\infty$, we use PL approximation~\cite{Mois77} and the two-dimensional version of the smoothing theorem~\cite[Theorem~6.3 and Section~9]{Munk60} and choose a smooth approximation $a_\ell$
\begin{equation} \label{eq:surface_smoothing}
	a_\ell\in\mathrm{Diff}^+(\M,\N), \qquad
	\|a_\ell-f\|_{L^\infty(\M)} \to 0.
\end{equation}
We use the $C^\infty$ version so that the Jacobians used in the subsequent Moser correction are smooth. We then estimate its cell-average area measure.

\begin{lemma} \label{lem:area_cutoff}
	Let $\mathcal{P}_\delta(\M)$ be a finite cell decomposition of $\M$ whose maximum cell diameter is less than $\delta>0$. For every $\eta>0$,
	\begin{enumerate}
		\item for all sufficiently large $\ell$,
		\begin{equation} \label{eq:area_smooth_avg}
		\|\bar{\mu}_{a_\ell}-\mathrm{C}_\delta\widehat J_f\|_{L^2(\M)}<\eta, \qquad
		\bar{\mu}_{a_\ell}(x)
		:=\frac{\mu_{a_\ell}(Q)}{|Q|}, \quad  x\in Q\in\mathcal{P}_\delta(\M).
		\end{equation}
		\item let $\chi_\varepsilon\in C_c^\infty(\operatorname{int}Q)$ for $Q\in\mathcal{P}_\delta(\M)$ such that 	$0\leq\chi_\varepsilon\leq1$ and $\chi_\varepsilon\to1$ pointwise on $\operatorname{int}Q$ as $\varepsilon\to0$, and define
		\[
		c_{Q,\varepsilon}
		:=\frac{\int_Q\chi_\varepsilon J_{a_\ell}\,\d A_\M}
		{\int_Q\chi_\varepsilon\,\d A_\M},
		\qquad
		\sigma_{Q,\varepsilon}
		:=(1-\chi_\varepsilon)J_{a_\ell}
		+\chi_\varepsilon c_{Q,\varepsilon}.
		\]
		Then, $\sigma_{Q,\varepsilon}$ is smooth and satisfies
		\begin{equation} \label{eq:area_cutoff}
		\int_Q\sigma_{Q,\varepsilon}\,\d A_\M=\mu_{a_\ell}(Q),
		\qquad
		\big\| \sigma_{Q,\varepsilon} -  \left.\bar{\mu}_{a_\ell}\right|_Q \big\|_{L^2(Q)} < \eta
	\end{equation}
		for sufficiently small $\varepsilon$.
	\end{enumerate}
\end{lemma}

\begin{proof}
By Lemma~\ref{lem:weak_conv_area}, $\mu_{a_\ell}\rightharpoonup\mu_f$ weakly. Since each cell boundary satisfies $|\partial Q|_\M=0$, we have $\mu_f(\partial Q)=0$, and by the Portmanteau Theorem,
\begin{equation} \label{eq:area_cutoff_1}
\mu_{a_\ell}(Q)\longrightarrow\mu_f(Q),
\qquad Q\in\mathcal{P}_\delta(\M).
\end{equation}
Therefore,
\[
	\|\bar{\mu}_{a_\ell}-\mathrm{C}_\delta\widehat J_f\|_{L^2(\M)}^2
	=\sum_{Q\in\mathcal{P}_\delta(\M)}
	\int_Q\left| \frac{\mu_{a_\ell}(Q)-\mu_f(Q)}{|Q|} \right|^2\,\d A_\M
\]
tends to $0$ as $\ell\to\infty$, which shows \eqref{eq:area_smooth_avg}.

Next, since $\chi_\varepsilon$ vanishes near $\partial Q$, we have $\sigma_{Q, \varepsilon} =J_{a_\ell}$ near the boundary and the cellwise densities glue smoothly across adjacent cells. Its total mass satisfies
\[
\int_Q\sigma_{Q,\varepsilon}\,\d A_\M
=\int_Q(1-\chi_\varepsilon)J_{a_\ell}\,\d A_\M
+c_{Q,\varepsilon}\int_Q\chi_\varepsilon\,\d A_\M
=\int_QJ_{a_\ell}\,\d A_\M
=\mu_{a_\ell}(Q).
\]
Also, since $0\leq\chi_\varepsilon\leq1$, the dominated convergence theorem gives
\[
\int_Q\chi_\varepsilon J_{a_\ell}\,\d A_\M \longrightarrow\int_QJ_{a_\ell}\,\d A_\M,
\qquad
\int_Q\chi_\varepsilon\,\d A_\M\longrightarrow |Q|.
\]
Hence,
\[
c_{Q,\varepsilon} \longrightarrow \left.\bar{\mu}_{a_\ell}\right|_Q, \qquad \sigma_{Q,\varepsilon}(x) \longrightarrow\left.\bar{\mu}_{a_\ell}\right|_Q
\]
for every $x \in \operatorname{int}Q$. Consequently, since $J_{a_\ell}$ is bounded on the compact surface $\M$, the dominated convergence theorem gives
\[
\int_Q\left|
\sigma_{Q,\varepsilon}
-\left.\bar{\mu}_{a_\ell}\right|_Q
\right|^2\,\d A_\M < \eta^2
\]
for sufficiently small $\varepsilon$.
\end{proof}

Now, with smooth $\sigma$, we may use a local Moser flow in each cell to realize the cell averages $\mathrm C_\delta\widehat J_f$ by a diffeomorphism close to $f$ (Figure~\ref{fig:cell_density}). We write
\[
\omega_f(\delta):=\sup_{d_\M(y,z)\leq\delta}|f(y)-f(z)|
\]
and provide the following lemma.

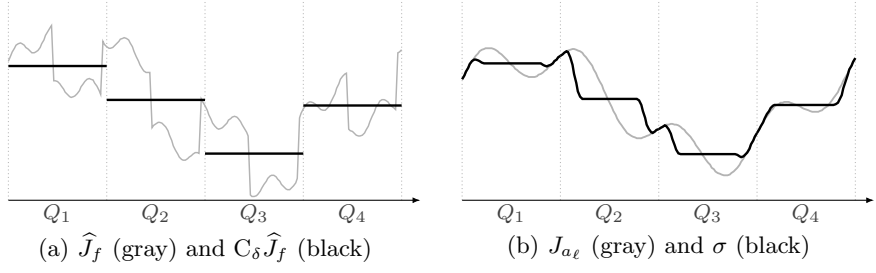
\begin{figure}[tb]
\centering
\begin{tikzpicture}
\begin{scope}[xshift=0.00cm]
\draw[gray!55, densely dotted, line width=0.4pt] (0.000,0) -- (0.000,2.640);
\draw[gray!55, densely dotted, line width=0.4pt] (1.300,0) -- (1.300,2.640);
\draw[gray!55, densely dotted, line width=0.4pt] (2.600,0) -- (2.600,2.640);
\draw[gray!55, densely dotted, line width=0.4pt] (3.900,0) -- (3.900,2.640);
\draw[gray!55, densely dotted, line width=0.4pt] (5.200,0) -- (5.200,2.640);
\draw[-{Latex[length=3pt]}, line width=0.4pt] (0,0) -- (5.450,0);
\draw[gray!60, line width=0.55pt] (0.00,1.84) -- (0.03,1.91) -- (0.05,1.96) -- (0.08,2.01) -- (0.10,2.05) -- (0.13,2.07) -- (0.16,2.08) -- (0.18,2.07) -- (0.21,2.05) -- (0.23,2.02) -- (0.26,2.00) -- (0.29,1.97) -- (0.31,1.95) -- (0.34,1.95) -- (0.36,1.96) -- (0.39,1.98) -- (0.42,2.02) -- (0.44,2.07) -- (0.47,2.12) -- (0.49,2.18) -- (0.52,2.23) -- (0.55,2.28) -- (0.57,2.31) -- (0.60,1.58) -- (0.62,1.58) -- (0.65,1.57) -- (0.68,1.54) -- (0.70,1.51) -- (0.73,1.47) -- (0.75,1.43) -- (0.78,1.40) -- (0.81,1.37) -- (0.83,1.36) -- (0.86,1.37) -- (0.88,1.38) -- (0.91,1.41) -- (0.94,1.45) -- (0.96,1.50) -- (0.99,1.54) -- (1.01,1.57) -- (1.04,1.60) -- (1.07,1.61) -- (1.09,1.60) -- (1.12,1.59) -- (1.14,1.55) -- (1.17,1.51) -- (1.20,1.45) -- (1.22,1.40) -- (1.25,2.10) -- (1.27,2.06) -- (1.30,2.03) -- (1.33,2.01) -- (1.35,2.01) -- (1.38,2.02) -- (1.40,2.04) -- (1.43,2.07) -- (1.46,2.10) -- (1.48,2.13) -- (1.51,2.14) -- (1.53,2.15) -- (1.56,2.14) -- (1.59,2.12) -- (1.61,2.08) -- (1.64,2.03) -- (1.66,1.97) -- (1.69,1.90) -- (1.72,1.84) -- (1.74,1.79) -- (1.77,1.74) -- (1.79,1.71) -- (1.82,1.69) -- (1.85,1.69) -- (1.87,1.69) -- (1.90,0.96) -- (1.92,0.99) -- (1.95,1.01) -- (1.98,1.03) -- (2.00,1.03) -- (2.03,1.02) -- (2.05,1.00) -- (2.08,0.97) -- (2.11,0.92) -- (2.13,0.86) -- (2.16,0.79) -- (2.18,0.72) -- (2.21,0.66) -- (2.24,0.61) -- (2.26,0.57) -- (2.29,0.54) -- (2.31,0.53) -- (2.34,0.53) -- (2.37,0.55) -- (2.39,0.57) -- (2.42,0.59) -- (2.44,0.61) -- (2.47,0.63) -- (2.50,0.63) -- (2.52,0.62) -- (2.55,1.34) -- (2.57,1.30) -- (2.60,1.25) -- (2.63,1.19) -- (2.65,1.13) -- (2.68,1.07) -- (2.70,1.02) -- (2.73,0.98) -- (2.76,0.95) -- (2.78,0.94) -- (2.81,0.95) -- (2.83,0.96) -- (2.86,0.99) -- (2.89,1.02) -- (2.91,1.06) -- (2.94,1.08) -- (2.96,1.10) -- (2.99,1.11) -- (3.02,1.10) -- (3.04,1.08) -- (3.07,1.04) -- (3.09,1.00) -- (3.12,0.95) -- (3.15,0.90) -- (3.17,0.85) -- (3.20,0.07) -- (3.22,0.05) -- (3.25,0.05) -- (3.28,0.06) -- (3.30,0.08) -- (3.33,0.12) -- (3.35,0.16) -- (3.38,0.21) -- (3.41,0.25) -- (3.43,0.29) -- (3.46,0.32) -- (3.48,0.33) -- (3.51,0.33) -- (3.54,0.31) -- (3.56,0.28) -- (3.59,0.25) -- (3.61,0.21) -- (3.64,0.18) -- (3.67,0.15) -- (3.69,0.14) -- (3.72,0.14) -- (3.74,0.16) -- (3.77,0.19) -- (3.80,0.23) -- (3.82,0.29) -- (3.85,1.09) -- (3.87,1.15) -- (3.90,1.21) -- (3.93,1.25) -- (3.95,1.28) -- (3.98,1.30) -- (4.00,1.30) -- (4.03,1.29) -- (4.06,1.27) -- (4.08,1.25) -- (4.11,1.22) -- (4.13,1.20) -- (4.16,1.19) -- (4.19,1.19) -- (4.21,1.21) -- (4.24,1.24) -- (4.26,1.29) -- (4.29,1.35) -- (4.32,1.41) -- (4.34,1.48) -- (4.37,1.54) -- (4.39,1.60) -- (4.42,1.64) -- (4.45,1.67) -- (4.47,1.68) -- (4.50,0.94) -- (4.52,0.92) -- (4.55,0.91) -- (4.58,0.88) -- (4.60,0.86) -- (4.63,0.85) -- (4.65,0.85) -- (4.68,0.86) -- (4.71,0.89) -- (4.73,0.93) -- (4.76,0.98) -- (4.78,1.04) -- (4.81,1.11) -- (4.84,1.17) -- (4.86,1.23) -- (4.89,1.27) -- (4.91,1.31) -- (4.94,1.33) -- (4.97,1.33) -- (4.99,1.32) -- (5.02,1.30) -- (5.04,1.27) -- (5.07,1.24) -- (5.10,1.22) -- (5.12,1.20) -- (5.15,1.95) -- (5.17,1.96) -- (5.20,1.99);
\draw[black, line width=0.9pt] (0.000,1.775) -- (1.300,1.775);
\draw[black, line width=0.9pt] (1.300,1.327) -- (2.600,1.327);
\draw[black, line width=0.9pt] (2.600,0.617) -- (3.900,0.617);
\draw[black, line width=0.9pt] (3.900,1.253) -- (5.200,1.253);
\node[font=\scriptsize, text=gray!50!black] at (0.650,-0.18) {$Q_1$};
\node[font=\scriptsize, text=gray!50!black] at (1.950,-0.18) {$Q_2$};
\node[font=\scriptsize, text=gray!50!black] at (3.250,-0.18) {$Q_3$};
\node[font=\scriptsize, text=gray!50!black] at (4.550,-0.18) {$Q_4$};
\node[font=\footnotesize] at (2.600,-0.62) {(a) $\widehat J_f$ (gray) and $\mathrm C_\delta\widehat J_f$ (black)};
\end{scope}
\begin{scope}[xshift=6.00cm]
\draw[gray!55, densely dotted, line width=0.4pt] (0.000,0) -- (0.000,2.640);
\draw[gray!55, densely dotted, line width=0.4pt] (1.300,0) -- (1.300,2.640);
\draw[gray!55, densely dotted, line width=0.4pt] (2.600,0) -- (2.600,2.640);
\draw[gray!55, densely dotted, line width=0.4pt] (3.900,0) -- (3.900,2.640);
\draw[gray!55, densely dotted, line width=0.4pt] (5.200,0) -- (5.200,2.640);
\draw[-{Latex[length=3pt]}, line width=0.4pt] (0,0) -- (5.450,0);
\draw[gray!60, line width=0.7pt] (0.00,1.60) -- (0.03,1.65) -- (0.05,1.70) -- (0.08,1.75) -- (0.10,1.80) -- (0.13,1.84) -- (0.16,1.88) -- (0.18,1.91) -- (0.21,1.94) -- (0.23,1.97) -- (0.26,1.99) -- (0.29,2.00) -- (0.31,2.01) -- (0.34,2.01) -- (0.36,2.01) -- (0.39,2.01) -- (0.42,2.00) -- (0.44,1.99) -- (0.47,1.97) -- (0.49,1.95) -- (0.52,1.93) -- (0.55,1.90) -- (0.57,1.88) -- (0.60,1.85) -- (0.62,1.82) -- (0.65,1.79) -- (0.68,1.77) -- (0.70,1.74) -- (0.73,1.72) -- (0.75,1.70) -- (0.78,1.68) -- (0.81,1.66) -- (0.83,1.65) -- (0.86,1.64) -- (0.88,1.63) -- (0.91,1.63) -- (0.94,1.63) -- (0.96,1.63) -- (0.99,1.64) -- (1.01,1.66) -- (1.04,1.67) -- (1.07,1.69) -- (1.09,1.71) -- (1.12,1.73) -- (1.14,1.76) -- (1.17,1.78) -- (1.20,1.81) -- (1.22,1.84) -- (1.25,1.86) -- (1.27,1.89) -- (1.30,1.91) -- (1.33,1.93) -- (1.35,1.95) -- (1.38,1.97) -- (1.40,1.98) -- (1.43,1.99) -- (1.46,2.00) -- (1.48,2.00) -- (1.51,1.99) -- (1.53,1.98) -- (1.56,1.97) -- (1.59,1.95) -- (1.61,1.92) -- (1.64,1.89) -- (1.66,1.86) -- (1.69,1.82) -- (1.72,1.78) -- (1.74,1.73) -- (1.77,1.68) -- (1.79,1.63) -- (1.82,1.57) -- (1.85,1.52) -- (1.87,1.46) -- (1.90,1.40) -- (1.92,1.35) -- (1.95,1.29) -- (1.98,1.23) -- (2.00,1.18) -- (2.03,1.13) -- (2.05,1.08) -- (2.08,1.03) -- (2.11,0.99) -- (2.13,0.96) -- (2.16,0.92) -- (2.18,0.89) -- (2.21,0.87) -- (2.24,0.85) -- (2.26,0.84) -- (2.29,0.83) -- (2.31,0.82) -- (2.34,0.82) -- (2.37,0.82) -- (2.39,0.83) -- (2.42,0.84) -- (2.44,0.85) -- (2.47,0.86) -- (2.50,0.88) -- (2.52,0.89) -- (2.55,0.91) -- (2.57,0.93) -- (2.60,0.94) -- (2.63,0.96) -- (2.65,0.97) -- (2.68,0.99) -- (2.70,1.00) -- (2.73,1.00) -- (2.76,1.01) -- (2.78,1.01) -- (2.81,1.00) -- (2.83,1.00) -- (2.86,0.99) -- (2.89,0.97) -- (2.91,0.95) -- (2.94,0.93) -- (2.96,0.90) -- (2.99,0.87) -- (3.02,0.84) -- (3.04,0.81) -- (3.07,0.77) -- (3.09,0.73) -- (3.12,0.69) -- (3.15,0.65) -- (3.17,0.61) -- (3.20,0.57) -- (3.22,0.54) -- (3.25,0.50) -- (3.28,0.47) -- (3.30,0.43) -- (3.33,0.41) -- (3.35,0.38) -- (3.38,0.36) -- (3.41,0.35) -- (3.43,0.34) -- (3.46,0.33) -- (3.48,0.33) -- (3.51,0.33) -- (3.54,0.34) -- (3.56,0.36) -- (3.59,0.38) -- (3.61,0.40) -- (3.64,0.43) -- (3.67,0.47) -- (3.69,0.51) -- (3.72,0.55) -- (3.74,0.59) -- (3.77,0.64) -- (3.80,0.69) -- (3.82,0.74) -- (3.85,0.79) -- (3.87,0.84) -- (3.90,0.89) -- (3.93,0.94) -- (3.95,0.98) -- (3.98,1.03) -- (4.00,1.07) -- (4.03,1.11) -- (4.06,1.15) -- (4.08,1.18) -- (4.11,1.21) -- (4.13,1.24) -- (4.16,1.26) -- (4.19,1.27) -- (4.21,1.29) -- (4.24,1.29) -- (4.26,1.30) -- (4.29,1.30) -- (4.32,1.30) -- (4.34,1.29) -- (4.37,1.28) -- (4.39,1.27) -- (4.42,1.26) -- (4.45,1.25) -- (4.47,1.24) -- (4.50,1.22) -- (4.52,1.21) -- (4.55,1.20) -- (4.58,1.19) -- (4.60,1.18) -- (4.63,1.17) -- (4.65,1.17) -- (4.68,1.17) -- (4.71,1.17) -- (4.73,1.18) -- (4.76,1.19) -- (4.78,1.21) -- (4.81,1.23) -- (4.84,1.25) -- (4.86,1.28) -- (4.89,1.32) -- (4.91,1.35) -- (4.94,1.39) -- (4.97,1.43) -- (4.99,1.48) -- (5.02,1.53) -- (5.04,1.58) -- (5.07,1.63) -- (5.10,1.68) -- (5.12,1.73) -- (5.15,1.78) -- (5.17,1.83) -- (5.20,1.88);
\draw[black, line width=0.9pt] (0.00,1.60) -- (0.03,1.65) -- (0.05,1.70) -- (0.08,1.75) -- (0.10,1.80) -- (0.13,1.84) -- (0.16,1.86) -- (0.18,1.86) -- (0.21,1.85) -- (0.23,1.84) -- (0.26,1.82) -- (0.29,1.81) -- (0.31,1.81) -- (0.34,1.81) -- (0.36,1.81) -- (0.39,1.81) -- (0.42,1.81) -- (0.44,1.81) -- (0.47,1.81) -- (0.49,1.81) -- (0.52,1.81) -- (0.55,1.81) -- (0.57,1.81) -- (0.60,1.81) -- (0.62,1.81) -- (0.65,1.81) -- (0.68,1.81) -- (0.70,1.81) -- (0.73,1.81) -- (0.75,1.81) -- (0.78,1.81) -- (0.81,1.81) -- (0.83,1.81) -- (0.86,1.81) -- (0.88,1.81) -- (0.91,1.81) -- (0.94,1.81) -- (0.96,1.81) -- (0.99,1.81) -- (1.01,1.81) -- (1.04,1.81) -- (1.07,1.79) -- (1.09,1.78) -- (1.12,1.77) -- (1.14,1.78) -- (1.17,1.79) -- (1.20,1.81) -- (1.22,1.84) -- (1.25,1.86) -- (1.27,1.89) -- (1.30,1.91);
\draw[black, line width=0.9pt] (1.30,1.91) -- (1.33,1.93) -- (1.35,1.95) -- (1.38,1.97) -- (1.40,1.96) -- (1.43,1.89) -- (1.46,1.79) -- (1.48,1.67) -- (1.51,1.55) -- (1.53,1.44) -- (1.56,1.37) -- (1.59,1.34) -- (1.61,1.34) -- (1.64,1.34) -- (1.66,1.34) -- (1.69,1.34) -- (1.72,1.34) -- (1.74,1.34) -- (1.77,1.34) -- (1.79,1.34) -- (1.82,1.34) -- (1.85,1.34) -- (1.87,1.34) -- (1.90,1.34) -- (1.92,1.34) -- (1.95,1.34) -- (1.98,1.34) -- (2.00,1.34) -- (2.03,1.34) -- (2.05,1.34) -- (2.08,1.34) -- (2.11,1.34) -- (2.13,1.34) -- (2.16,1.34) -- (2.18,1.34) -- (2.21,1.34) -- (2.24,1.34) -- (2.26,1.34) -- (2.29,1.34) -- (2.31,1.34) -- (2.34,1.32) -- (2.37,1.26) -- (2.39,1.18) -- (2.42,1.09) -- (2.44,1.00) -- (2.47,0.94) -- (2.50,0.90) -- (2.52,0.89) -- (2.55,0.91) -- (2.57,0.93) -- (2.60,0.94);
\draw[black, line width=0.9pt] (2.60,0.94) -- (2.63,0.96) -- (2.65,0.97) -- (2.68,0.99) -- (2.70,0.98) -- (2.73,0.94) -- (2.76,0.88) -- (2.78,0.81) -- (2.81,0.73) -- (2.83,0.67) -- (2.86,0.62) -- (2.89,0.61) -- (2.91,0.61) -- (2.94,0.61) -- (2.96,0.61) -- (2.99,0.61) -- (3.02,0.61) -- (3.04,0.61) -- (3.07,0.61) -- (3.09,0.61) -- (3.12,0.61) -- (3.15,0.61) -- (3.17,0.61) -- (3.20,0.61) -- (3.22,0.61) -- (3.25,0.61) -- (3.28,0.61) -- (3.30,0.61) -- (3.33,0.61) -- (3.35,0.61) -- (3.38,0.61) -- (3.41,0.61) -- (3.43,0.61) -- (3.46,0.61) -- (3.48,0.61) -- (3.51,0.61) -- (3.54,0.61) -- (3.56,0.61) -- (3.59,0.61) -- (3.61,0.61) -- (3.64,0.60) -- (3.67,0.58) -- (3.69,0.57) -- (3.72,0.58) -- (3.74,0.60) -- (3.77,0.63) -- (3.80,0.68) -- (3.82,0.74) -- (3.85,0.79) -- (3.87,0.84) -- (3.90,0.89);
\draw[black, line width=0.9pt] (3.90,0.89) -- (3.93,0.94) -- (3.95,0.98) -- (3.98,1.03) -- (4.00,1.08) -- (4.03,1.14) -- (4.06,1.18) -- (4.08,1.22) -- (4.11,1.25) -- (4.13,1.26) -- (4.16,1.26) -- (4.19,1.26) -- (4.21,1.26) -- (4.24,1.26) -- (4.26,1.26) -- (4.29,1.26) -- (4.32,1.26) -- (4.34,1.26) -- (4.37,1.26) -- (4.39,1.26) -- (4.42,1.26) -- (4.45,1.26) -- (4.47,1.26) -- (4.50,1.26) -- (4.52,1.26) -- (4.55,1.26) -- (4.58,1.26) -- (4.60,1.26) -- (4.63,1.26) -- (4.65,1.26) -- (4.68,1.26) -- (4.71,1.26) -- (4.73,1.26) -- (4.76,1.26) -- (4.78,1.26) -- (4.81,1.26) -- (4.84,1.26) -- (4.86,1.26) -- (4.89,1.26) -- (4.91,1.26) -- (4.94,1.27) -- (4.97,1.29) -- (4.99,1.33) -- (5.02,1.39) -- (5.04,1.48) -- (5.07,1.57) -- (5.10,1.66) -- (5.12,1.73) -- (5.15,1.78) -- (5.17,1.83) -- (5.20,1.88);
\node[font=\scriptsize, text=gray!50!black] at (0.650,-0.18) {$Q_1$};
\node[font=\scriptsize, text=gray!50!black] at (1.950,-0.18) {$Q_2$};
\node[font=\scriptsize, text=gray!50!black] at (3.250,-0.18) {$Q_3$};
\node[font=\scriptsize, text=gray!50!black] at (4.550,-0.18) {$Q_4$};
\node[font=\footnotesize] at (2.600,-0.62) {(b) $J_{a_\ell}$ (gray) and $\sigma$ (black)};
\end{scope}
\end{tikzpicture}
\caption{The densities of Lemmas~\ref{lem:cell_decompose}--\ref{lem:cell_realize} on four cells: (a) $\widehat J_f$ and its cell averages; (b) $J_{a_\ell}$ and $\sigma$, which keeps the cell masses and equals $J_{a_\ell}$ near the cell boundaries.}
\label{fig:cell_density}
\end{figure}

\begin{lemma}\label{lem:cell_realize}
For each $\delta>0$, let $\mathcal{P}_\delta(\M)$ be a finite cell decomposition of $\M$ whose maximum diameter is less than $\delta$. For $\eta > 0$, there exists $g\in\mathrm{Diff}^+(\M,\N)$, smooth up to the boundary, such that
\[
\|J_g-\mathrm C_\delta\widehat J_f\|_{L^2(\M)}<\eta,
\qquad
\|g-f\|_{L^\infty(\M)}<\eta+\omega_f(\delta).
\]
Moreover, $|g(x)-f(x)|<\eta$ for every $x$ on the boundary of a cell.
\end{lemma}

\begin{proof}
Write $\mathcal{P}_\delta(\M)=\{Q_1,\ldots,Q_N\}$. By Lemma~\ref{lem:area_cutoff} and \eqref{eq:surface_smoothing}, there exist an index $\ell$ with $\|a_\ell-f\|_{L^\infty(\M)}<\eta$ and $\varepsilon_k$ such that
\begin{equation} \label{eq:cell_realize_1}
\|\bar{\mu}_{a_\ell}-\mathrm{C}_\delta\widehat J_f\|_{L^2(\M)}<\frac{\eta}{2},\qquad
\|\sigma_{Q_k,\varepsilon_k}-\bar{\mu}_{a_\ell}\|_{L^2(Q_k)}
<\frac{\eta}{2\sqrt{N}}.
\end{equation}
Define a global density $\sigma$ by $\sigma|_{Q_k} = \sigma_{Q_k,\varepsilon_k}$. Since each $\sigma_{Q_k,\varepsilon_k}$ is smooth and agrees with $J_{a_\ell}$ near the cell boundary, $\sigma$ is smooth globally on $\M$ and
\begin{equation} \label{eq:cell_realize_2}
\|\sigma-\bar{\mu}_{a_\ell}\|_{L^2(\M)}^2
=\sum_{k=1}^N
\|\sigma_{Q_k,\varepsilon_k}-\bar{\mu}_{a_\ell}\|_{L^2(Q_k)}^2
<\frac{\eta^2}{4}.
\end{equation}

In addition, by \eqref{eq:area_cutoff}, $\sigma$ has the same total area as $J_{a_\ell}$  on each ${Q_k}$ and agrees with it near $\partial Q_k$.  Moser and Banyaga's argument \cite{Mose65, Bany1974} gives a diffeomorphism $\phi_Q:Q\to Q$, equal to the identity near
$\partial Q$, such that
\[
\phi_Q^*(J_{a_\ell}\,\d A_\M)=\sigma\,\d A_\M
\]
and the maps $\phi_Q$ glue to a global diffeomorphism $\phi:\M\to\M$ satisfying $\phi|_Q=\phi_Q$ for every $Q\in\mathcal{P}_\delta(\M)$.

Therefore, by defining $g:=a_\ell\circ\phi$, we obtain
\[
g^*\d A_\N
=\phi^*(a_\ell^*\d A_\N)
=\phi^*(J_{a_\ell}\,\d A_\M)
=\sigma\,\d A_\M,
\]
so $J_g = \sigma$, and by \eqref{eq:cell_realize_1} and \eqref{eq:cell_realize_2}, we obtain
\[
\|J_g-\mathrm C_\delta\widehat J_f\|_{L^2(\M)}
\leq\|\sigma-\bar{\mu}_{a_\ell}\|_{L^2(\M)}
+\|\bar{\mu}_{a_\ell}-\mathrm{C}_\delta\widehat J_f\|_{L^2(\M)}<\eta .
\]
Finally, if $x\in Q$, then $x,\phi(x)\in Q$ and hence $d_\M(x,\phi(x))<\delta$. Thus \eqref{eq:surface_smoothing} gives
\[
|g(x)-f(x)| \leq|a_\ell(\phi(x))-f(\phi(x))| +|f(\phi(x))-f(x)|
<\eta+\omega_f(\delta).
\]
If, in addition, $x\in\partial Q$, then $\phi(x)=x$, and hence $|g(x)-f(x)|=|a_\ell(x)-f(x)|<\eta$.
\end{proof}

Letting the cells shrink yields the main approximation result.

\begin{theorem} \label{thm:energy_dense}
For every $f\in\mathrm{Homeo}^+(\M,\N)$ with $\widebar{\E_\rS}(f)<\infty$, there exists a sequence $\{g_j\}_{j=1}^\infty\subset\mathrm{Diff}^+(\M,\N)$, smooth up to the
boundary, such that
\begin{equation} \label{eq:energy_dense}
\|g_j-f\|_{L^\infty(\M)}\longrightarrow0,
\qquad
\|J_{g_j}-\widehat J_f\|_{L^2(\M)}\longrightarrow0.
\end{equation}
In particular, $\E_\rS(g_j)\to\widebar{\E_\rS}(f)$.
\end{theorem}

\begin{proof}
Choose $\delta_j\downarrow0$ and a sequence of cell decompositions $\mathcal P_{\delta_j}(\M)$ whose maximum diameter is less than $\delta_j$. Lemma~\ref{lem:cell_realize} with $\eta=\delta_j$ gives $g_j\in\mathrm{Diff}^+(\M,\N)$ with
\begin{equation} \label{eq:energy_dense_1}
\|J_{g_j}-\mathrm C_{\delta_j}\widehat J_f\|_{L^2(\M)}<\delta_j,
\qquad
\|g_j-f\|_{L^\infty(\M)}<\delta_j+\omega_f(\delta_j).
\end{equation}
By the triangle inequality, we have
\[
\|J_{g_j}-\widehat J_f\|_{L^2(\M)}
\leq \|J_{g_j} - \mathrm{C}_{\delta_j}\widehat J_f\|_{L^2(\M)}  +\|\mathrm{C}_{\delta_j}\widehat J_f-\widehat J_f\|_{L^2(\M)} \longrightarrow 0
\]
where the first and second terms vanish by \eqref{eq:energy_dense_1} and Lemma~\ref{lem:cell_decompose}, respectively.
Moreover, since $\omega_f(\delta_j) \to 0$, $\|g_j-f\|_{L^\infty(\M)} \to 0$.
Note that the convergence of the energies follows from the convergence $J_{g_j}\to\widehat J_f$ in $L^2(\M)$.
\end{proof}

Together with the lower semicontinuity in Lemma~\ref{lem:weak_conv_area}, Theorem~\ref{thm:energy_dense} identifies $\widebar{\E_\rS}$ as the relaxation of $\E_\rS$.

\begin{corollary}\label{cor:energy_relax}
For every $f\in\mathrm{Homeo}^+(\M,\N)$,
\begin{equation} \label{eq:energy_relax}
\widebar{\E_\rS}(f) =
\inf \left\{ \liminf_{\ell\to\infty}\E_\rS(g_\ell)
\,\middle|\,
g_\ell\in\mathrm{Diff}^+(\M,\N), \quad d_H(g_\ell,f)\to0 \right\}.
\end{equation}
\end{corollary}

\begin{proof}
Let $\{g_\ell\}\subset\mathrm{Diff}^+(\M,\N)$ satisfy $d_H(g_\ell,f)\to0$. The lower semicontinuity in Lemma~\ref{lem:weak_conv_area} gives
\begin{equation} \label{eq:energy_relax_1}
\widebar{\E_\rS}(f)
\leq\liminf_{\ell\to\infty}\widebar{\E_\rS}(g_\ell)
=\liminf_{\ell\to\infty}\E_\rS(g_\ell).
\end{equation}
For the reverse inequality, if $\widebar{\E_\rS}(f)=+\infty$, \eqref{eq:energy_relax_1} forces the right-hand side to be $+\infty$. Suppose $\widebar{\E_\rS}(f)<+\infty$. By Theorem~\ref{thm:energy_dense}, there exists a sequence $\{g_\ell\} \subset \mathrm{Diff}^+(\M,\N)$ such that
\[
d_H(g_\ell,f)\to0,
\qquad
\|J_{g_\ell}-\widehat J_f\|_{L^2(\M)}\to0.
\]
Consequently,
\[
\liminf_{\ell \to \infty} \E_\rS(g_\ell) = \lim_{\ell \to \infty} \E_\rS(g_\ell) = \widebar{\E_\rS}(f),
\]
which gives the reverse inequality.
\end{proof}

For diffeomorphisms, a critical point of $\E_\rS$ under boundary-fixing variations is area-preserving \cite[Theorem~2]{LiKY26}; in particular, every local minimizer is a global minimizer. This relaxation of the stretch energy keeps the same property, as shown in the following.

\begin{lemma} \label{lem:local_global}
Let $f\in\mathrm{Homeo}^+(\M,\N)$, and suppose that there exists $\rho>0$ such that
\begin{equation} \label{eq:local_global}
\widebar{\E_\rS}(f)\leq\widebar{\E_\rS}(v)\qquad\text{for all }v\in\mathrm{Homeo}^+(\M,\N)\text{ with }d_H(v,f)<\rho.
\end{equation}
Then
\[
\mu_f=\dfrac{|\N|}{|\M|}\,\d A_\M.
\]
In particular, if $|\M| = |\N|$, then $f$ preserves area and $\widebar{\E_\rS}(f)=|\M|$.
\end{lemma}

\begin{proof}
Similar to \eqref{eq:surface_smoothing}, we use PL approximation~\cite{Mois77} and the two-dimensional version of the smoothing theorem \cite[Theorem 6.3 and Section 9]{Munk60} to choose a diffeomorphism $v$ such that $d_H(v,f)<\rho$. The local minimality of $f$ implies $\widebar{\E_\rS}(f)\leq \E_\rS(v) < + \infty$, and hence $\mu_f=\widehat J_f\,\d A_\M$ with $\widehat J_f\in L^2(\M)$.

Fix $\delta>0$ with $\omega_f(\delta)<\rho/2$, and let $\mathcal P_\delta(\M)$ be a finite cell decomposition of $\M$ whose maximum diameter is less than $\delta$. We claim that
\begin{equation} \label{eq:local_global_1}
	\widehat J_f =   \mathrm C_\delta\widehat J_f\qquad\text{almost everywhere}.
\end{equation}
Since it holds for every decomposition whose maximum cell diameter is less than $\delta$, $\widehat J_f$ is almost everywhere constant on every cell of such a decomposition in $\operatorname{int}\M$, and hence near every point of $\operatorname{int}\M$. By the connectedness of $\M$ and $\mu_f|_{\partial \M} = 0$, $\widehat J_f=c$ almost everywhere on $\M$ for a constant $c$ with
\[
c\,|\M|=\mu_f(\M)=|f(\M)|_\N=|\N|,
\]
and $c = 1$ if $|\M|=|\N|$.

We now prove the claim. Lemma~\ref{lem:cell_realize} with $\eta\in(0,\rho/2)$ gives a diffeomorphism $g$ such that
\begin{equation} \label{eq:local_global_2}
d_H(g,f)<\eta+\omega_f(\delta)<\rho,
\qquad
\|J_g-\mathrm C_\delta\widehat J_f\|_{L^2(\M)}<\eta.
\end{equation}
Since $d_H(g,f)<\rho$, the local minimality \eqref{eq:local_global} and \eqref{eq:local_global_2} give
\[
\| \widehat J_f \|_{L^2(\M)} = \widebar{\E_\rS}(f)^{1/2}
\leq \widebar{\E_\rS}(g)^{1/2} = \| J_g\|_{L^2(\M)}
\leq \|\mathrm C_\delta\widehat J_f\|_{L^2(\M)} + \eta.
\]
Letting $\eta \to 0$ gives
\begin{equation} \label{eq:local_global_3}
\| \widehat J_f \|_{L^2(\M)} \leq \|\mathrm C_\delta\widehat J_f\|_{L^2(\M)} .
\end{equation}
In addition, we expand:
\begin{align*}
\| \widehat J_f\|^2_{L^2(\M)} &= \| \mathrm C_\delta\widehat J_f + (\widehat J_f- \mathrm C_\delta\widehat J_f)\|^2_{L^2(\M)} \\
&= \| \mathrm C_\delta\widehat J_f\|^2_{L^2(\M)} + \|\widehat J_f- \mathrm C_\delta\widehat J_f\|^2_{L^2(\M)} + 2 \langle\mathrm C_\delta\widehat J_f,\; \widehat J_f- \mathrm C_\delta\widehat J_f \rangle_{L^2(\M)}.
\end{align*}
On each cell $Q$, the function $\mathrm{C}_\delta \widehat J_f$ is the constant average of $\widehat J_f$, so
\[
	\int_Q (\widehat J_f - \mathrm{C}_\delta \widehat J_f) \, \mathrm{C}_\delta \widehat J_f \, \d A_\M =
	\mathrm{C}_\delta \widehat J_f\big|_Q \int_Q \big(\widehat J_f  - \mathrm{C}_\delta \widehat J_f\big) \, \d A_\M = 0.
\]
Consequently, the cross term vanishes:
\[
\begin{aligned}
	\langle\mathrm C_\delta\widehat J_f,\; \widehat J_f- \mathrm C_\delta\widehat J_f \rangle_{L^2(\M)} &=
	\sum_{Q\in\mathcal P_\delta(\M)}\int_Q (\widehat J_f - \mathrm{C}_\delta \widehat J_f) \, \mathrm{C}_\delta \widehat J_f \, \d A_\M= 0,
\end{aligned}
\]
and we obtain
\begin{equation} \label{eq:local_global_4}
\| \widehat J_f\|^2_{L^2(\M)} = \| \mathrm C_\delta\widehat J_f\|^2_{L^2(\M)} + \|\widehat J_f- \mathrm C_\delta\widehat J_f\|^2_{L^2(\M)}.
\end{equation}
Substituting \eqref{eq:local_global_4} into the left-hand side of \eqref{eq:local_global_3} after squaring, we obtain
\[
\| \mathrm C_\delta\widehat J_f\|^2_{L^2(\M)} + \|\widehat J_f- \mathrm C_\delta\widehat J_f\|^2_{L^2(\M)}
\leq \| \mathrm C_\delta\widehat J_f\|^2_{L^2(\M)},
\]
and hence
\[
\|\widehat J_f- \mathrm C_\delta\widehat J_f\|_{L^2(\M)}=0,
\]
which finishes the claim.
\end{proof}

Note that this extends the result of \cite[Theorem~2]{LiKY26} from diffeomorphisms to homeomorphisms: local minimality is global.

\section{Lifting discrete maps}
\label{sec:lift}

We now turn to the discrete setting.
Let $\Mh$ be a triangular mesh approximating $\M$ with maximal triangle diameter $h$:
\[
h:= \max_{\tau:\mathrm{triangles}\in \Mh} \mathrm{diam}(\tau),
\]
which we call the mesh size.

We approximate a homeomorphism $f:\M \to\N$ by a piecewise-linear interpolant $f_h$, called a simplicial map.
It is linear on each triangle, so $\Nh:= f_h(\Mh)$ is also a triangular mesh approximating $\N$.
In particular, given a triangle $\tau$ with vertices $v_0,v_1,v_2$, the image $f_h(x)$ for $x \in \tau$ is explicitly determined by the barycentric representation
\[
f_h(x)
= \lambda_0 (x)\,f(v_0) + \lambda_1 (x)\,f(v_1) + \lambda_2 (x)\,f(v_2),
\]
where the barycentric coordinates are
\begin{equation}
\lambda_0 (x)=\frac{\mathrm{area}([x,v_1,v_2])}{\mathrm{area}(\tau)} , \quad
\lambda_1 (x)=\frac{\mathrm{area}([v_0,x,v_2])}{\mathrm{area}(\tau)}, \quad
\lambda_2 (x)=\frac{\mathrm{area}([v_0,v_1,x])}{\mathrm{area}(\tau)}.
\label{eq:bary_coord}
\end{equation}
In this section, we provide a lifting that transforms a simplicial map $f_h$ into a homeomorphism and then estimate the error produced by such a lifting.

\subsection{Closest-point projection}
\label{sec:close}
Suppose $\partial\M=\varnothing$. By the Tubular Neighborhood Theorem~\cite[Theorem 6.24]{Lee13}, the tubular neighborhood of a smooth surface $\mathcal{M}$ is well-defined. We assume $\Mh$ lies within its tubular neighborhood and  use the closest-point projection \cite[Section 4.1]{DzEl13}
\begin{equation} \label{eq:close_point_proj}
\pi_\M(x) := \argmin_{p\in \M} \| p - x\|^2.
\end{equation}
Since $x$ is projected along the normal of $\M$ at $\pi_\M(x)$, if $\pi_\M$ is bijective, $x$ has
a unique representation
\begin{equation} \label{eq:patch_projection}
	x = \pi_\M(x) + t_\M(x) \,n_\M \big(\pi_\M(x) \big),
\end{equation}
where $t_\M$ is the signed distance function and $n_\M$ is the unit normal.

For $\partial\M \neq \varnothing$, we consider a smooth extension $\M^e\supset\M$ to ensure $\pi_\M$ is well defined on $\partial \Mh$. We keep the same notation $\pi_\M$ for $\pi_{\M^e}$ and use the assumption that $\pi_\M|_{\Mh}$ is bijective to obtain the same representation \eqref{eq:patch_projection}.
We further require shape regularity:
\begin{equation} \label{eq:shape_regular}
	\max_{\tau : \mathrm{triangles}} \frac{\mathrm{diam}(\tau)}{\mathrm{inrad}(\tau)} \leq C,
\end{equation}
where $\mathrm{diam}(\tau)$ and $\mathrm{inrad}(\tau)$ denote the diameter and in-ball radius of the triangle $\tau$. We therefore use the following assumption for $\Mh$.

\begin{assumption}[Source meshes]\label{ass:source}
	The triangulated surfaces $\{\Mh\}_{h >0}$ satisfy the following:
	\begin{itemize}
		\item[(i)] the vertices of $\Mh$ lie on $\M$, and the boundary vertices lie on $\partial\M$;
		\item[(ii)] $\pi_\M|_{\Mh}$ is a homeomorphism onto its image, which is $\M$ if $\partial\M=\varnothing$;
		\item[(iii)] $\{\Mh\}$ is uniformly shape-regular in the sense of \eqref{eq:shape_regular};
		\item[(iv)] if $\partial\M\neq\varnothing$, the arcs of $\partial\M$ joining the endpoints of the boundary edges of $\Mh$ cover $\partial\M$ exactly once.
	\end{itemize}
\end{assumption}

Assumption~\ref{ass:source} (i)--(iii) directly gives the estimate of area change under projection as in \cite[Lemma 4.1]{DzEl13}.

\begin{lemma}[{\cite[Lemma 4.1]{DzEl13}}] \label{lem:area_change}
	Suppose $\M$ is a smooth surface and $\Mh$ is a triangular mesh approximating $\M$ satisfying Assumption~\ref{ass:source}. Then, for sufficiently small $h$
	\begin{equation}\label{eq:proj_area_estimate}
		\big\| t_\M \big\|_{L^\infty(\Mh)} \leq Ch^2, \qquad
		\big\| 1 - J_{\pi_{\M}} \big\|_{L^\infty(\Mh)} \leq  C h^2.
	\end{equation}
\end{lemma}

To analyze the simplicial map, we can follow the same assumption with the closest-point projection $\pi_\N$. Therefore, $\pi_\N|_{\Nh}$ has the same estimate under the shape-regularity assumption. It is a valid assumption since the interpolant of a diffeomorphism satisfies shape-regularity for sufficiently small $h$.

Nonetheless, we are concerned with homeomorphisms in this work, and $E_\rS$ only controls area but not derivatives. We naturally extend the results to meshes without shape regularity, with only one-sided control (Figure~\ref{fig:thin_triangle}). To avoid confusion of notation, we denote the mesh size of the image mesh by
\[
k_h:= \max_{\tau:\mathrm{triangles}\in f_h(\Mh)} \mathrm{diam}(\tau)
\]
and provide the following lemma.

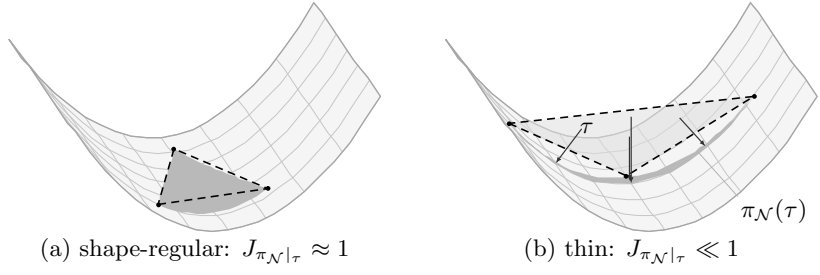
\begin{figure}[tb]
\centering
\begin{tikzpicture}[x=2.1cm, y=2.1cm, font=\footnotesize]
\fill[gray!8] (-0.75,0.33) -- (-0.66,0.20) -- (-0.56,0.08) -- (-0.47,-0.01) -- (-0.37,-0.10) -- (-0.27,-0.17) -- (-0.18,-0.22) -- (-0.08,-0.26) -- (0.01,-0.29) -- (0.11,-0.30) -- (0.21,-0.29) -- (0.30,-0.27) -- (0.40,-0.24) -- (0.50,-0.19) -- (0.59,-0.13) -- (0.69,-0.05) -- (0.78,0.04) -- (0.88,0.14) -- (0.98,0.26) -- (1.07,0.40) -- (1.17,0.55) -- (1.17,0.55) -- (1.13,0.61) -- (1.09,0.67) -- (1.04,0.73) -- (1.00,0.78) -- (0.96,0.84) -- (0.92,0.90) -- (0.88,0.96) -- (0.84,1.02) -- (0.80,1.08) -- (0.75,1.14) -- (0.75,1.14) -- (0.66,0.99) -- (0.56,0.85) -- (0.47,0.73) -- (0.37,0.62) -- (0.27,0.53) -- (0.18,0.46) -- (0.08,0.39) -- (-0.01,0.34) -- (-0.11,0.31) -- (-0.21,0.29) -- (-0.30,0.29) -- (-0.40,0.30) -- (-0.50,0.33) -- (-0.59,0.37) -- (-0.69,0.42) -- (-0.78,0.49) -- (-0.88,0.57) -- (-0.98,0.67) -- (-1.07,0.78) -- (-1.17,0.91) -- (-1.17,0.91) -- (-1.13,0.85) -- (-1.09,0.79) -- (-1.04,0.74) -- (-1.00,0.68) -- (-0.96,0.62) -- (-0.92,0.56) -- (-0.88,0.50) -- (-0.84,0.44) -- (-0.80,0.38) -- (-0.75,0.33) -- cycle;
\draw[gray!45, line width=0.3pt] (-0.75,0.33) -- (-0.66,0.20) -- (-0.56,0.08) -- (-0.47,-0.01) -- (-0.37,-0.10) -- (-0.27,-0.17) -- (-0.18,-0.22) -- (-0.08,-0.26) -- (0.01,-0.29) -- (0.11,-0.30) -- (0.21,-0.29) -- (0.30,-0.27) -- (0.40,-0.24) -- (0.50,-0.19) -- (0.59,-0.13) -- (0.69,-0.05) -- (0.78,0.04) -- (0.88,0.14) -- (0.98,0.26) -- (1.07,0.40) -- (1.17,0.55);
\draw[gray!45, line width=0.3pt] (-0.82,0.42) -- (-0.73,0.30) -- (-0.63,0.18) -- (-0.54,0.08) -- (-0.44,0.00) -- (-0.34,-0.07) -- (-0.25,-0.12) -- (-0.15,-0.16) -- (-0.05,-0.19) -- (0.04,-0.20) -- (0.14,-0.20) -- (0.23,-0.18) -- (0.33,-0.14) -- (0.43,-0.10) -- (0.52,-0.03) -- (0.62,0.04) -- (0.71,0.13) -- (0.81,0.24) -- (0.91,0.36) -- (1.00,0.50) -- (1.10,0.65);
\draw[gray!45, line width=0.3pt] (-0.89,0.52) -- (-0.80,0.39) -- (-0.70,0.28) -- (-0.60,0.18) -- (-0.51,0.10) -- (-0.41,0.03) -- (-0.32,-0.03) -- (-0.22,-0.07) -- (-0.12,-0.09) -- (-0.03,-0.10) -- (0.07,-0.10) -- (0.17,-0.08) -- (0.26,-0.05) -- (0.36,0.00) -- (0.45,0.06) -- (0.55,0.14) -- (0.65,0.23) -- (0.74,0.34) -- (0.84,0.46) -- (0.93,0.60) -- (1.03,0.75);
\draw[gray!45, line width=0.3pt] (-0.96,0.62) -- (-0.87,0.49) -- (-0.77,0.38) -- (-0.67,0.28) -- (-0.58,0.20) -- (-0.48,0.13) -- (-0.38,0.07) -- (-0.29,0.03) -- (-0.19,0.01) -- (-0.10,-0.00) -- (0.00,0.00) -- (0.10,0.02) -- (0.19,0.05) -- (0.29,0.10) -- (0.38,0.16) -- (0.48,0.24) -- (0.58,0.33) -- (0.67,0.44) -- (0.77,0.56) -- (0.87,0.69) -- (0.96,0.84);
\draw[gray!45, line width=0.3pt] (-1.03,0.72) -- (-0.93,0.59) -- (-0.84,0.48) -- (-0.74,0.38) -- (-0.65,0.29) -- (-0.55,0.22) -- (-0.45,0.17) -- (-0.36,0.13) -- (-0.26,0.10) -- (-0.17,0.09) -- (-0.07,0.10) -- (0.03,0.12) -- (0.12,0.15) -- (0.22,0.20) -- (0.32,0.26) -- (0.41,0.34) -- (0.51,0.43) -- (0.60,0.53) -- (0.70,0.66) -- (0.80,0.79) -- (0.89,0.94);
\draw[gray!45, line width=0.3pt] (-1.10,0.81) -- (-1.00,0.69) -- (-0.91,0.57) -- (-0.81,0.48) -- (-0.71,0.39) -- (-0.62,0.32) -- (-0.52,0.27) -- (-0.43,0.23) -- (-0.33,0.20) -- (-0.23,0.19) -- (-0.14,0.20) -- (-0.04,0.21) -- (0.05,0.25) -- (0.15,0.29) -- (0.25,0.36) -- (0.34,0.43) -- (0.44,0.53) -- (0.54,0.63) -- (0.63,0.75) -- (0.73,0.89) -- (0.82,1.04);
\draw[gray!45, line width=0.3pt] (-1.17,0.91) -- (-1.07,0.78) -- (-0.98,0.67) -- (-0.88,0.57) -- (-0.78,0.49) -- (-0.69,0.42) -- (-0.59,0.37) -- (-0.50,0.33) -- (-0.40,0.30) -- (-0.30,0.29) -- (-0.21,0.29) -- (-0.11,0.31) -- (-0.01,0.34) -- (0.08,0.39) -- (0.18,0.46) -- (0.27,0.53) -- (0.37,0.62) -- (0.47,0.73) -- (0.56,0.85) -- (0.66,0.99) -- (0.75,1.14);
\draw[gray!45, line width=0.3pt] (-0.75,0.33) -- (-0.80,0.38) -- (-0.84,0.44) -- (-0.88,0.50) -- (-0.92,0.56) -- (-0.96,0.62) -- (-1.00,0.68) -- (-1.04,0.74) -- (-1.09,0.79) -- (-1.13,0.85) -- (-1.17,0.91);
\draw[gray!45, line width=0.3pt] (-0.51,0.03) -- (-0.56,0.09) -- (-0.60,0.15) -- (-0.64,0.21) -- (-0.68,0.27) -- (-0.72,0.33) -- (-0.76,0.39) -- (-0.80,0.44) -- (-0.84,0.50) -- (-0.89,0.56) -- (-0.93,0.62);
\draw[gray!45, line width=0.3pt] (-0.27,-0.17) -- (-0.32,-0.11) -- (-0.36,-0.05) -- (-0.40,0.01) -- (-0.44,0.07) -- (-0.48,0.13) -- (-0.52,0.19) -- (-0.56,0.24) -- (-0.60,0.30) -- (-0.65,0.36) -- (-0.69,0.42);
\draw[gray!45, line width=0.3pt] (-0.03,-0.28) -- (-0.07,-0.22) -- (-0.12,-0.16) -- (-0.16,-0.10) -- (-0.20,-0.04) -- (-0.24,0.02) -- (-0.28,0.08) -- (-0.32,0.13) -- (-0.36,0.19) -- (-0.41,0.25) -- (-0.45,0.31);
\draw[gray!45, line width=0.3pt] (0.21,-0.29) -- (0.17,-0.23) -- (0.12,-0.18) -- (0.08,-0.12) -- (0.04,-0.06) -- (0.00,0.00) -- (-0.04,0.06) -- (-0.08,0.12) -- (-0.12,0.18) -- (-0.17,0.23) -- (-0.21,0.29);
\draw[gray!45, line width=0.3pt] (0.45,-0.22) -- (0.41,-0.16) -- (0.36,-0.10) -- (0.32,-0.04) -- (0.28,0.02) -- (0.24,0.07) -- (0.20,0.13) -- (0.16,0.19) -- (0.12,0.25) -- (0.07,0.31) -- (0.03,0.37);
\draw[gray!45, line width=0.3pt] (0.69,-0.05) -- (0.65,0.00) -- (0.60,0.06) -- (0.56,0.12) -- (0.52,0.18) -- (0.48,0.24) -- (0.44,0.30) -- (0.40,0.36) -- (0.36,0.41) -- (0.32,0.47) -- (0.27,0.53);
\draw[gray!45, line width=0.3pt] (0.93,0.20) -- (0.89,0.26) -- (0.84,0.32) -- (0.80,0.38) -- (0.76,0.44) -- (0.72,0.50) -- (0.68,0.55) -- (0.64,0.61) -- (0.60,0.67) -- (0.56,0.73) -- (0.51,0.79);
\draw[gray!45, line width=0.3pt] (1.17,0.55) -- (1.13,0.61) -- (1.09,0.67) -- (1.04,0.73) -- (1.00,0.78) -- (0.96,0.84) -- (0.92,0.90) -- (0.88,0.96) -- (0.84,1.02) -- (0.80,1.08) -- (0.75,1.14);
\draw[gray!70, line width=0.5pt] (-0.75,0.33) -- (-0.66,0.20) -- (-0.56,0.08) -- (-0.47,-0.01) -- (-0.37,-0.10) -- (-0.27,-0.17) -- (-0.18,-0.22) -- (-0.08,-0.26) -- (0.01,-0.29) -- (0.11,-0.30) -- (0.21,-0.29) -- (0.30,-0.27) -- (0.40,-0.24) -- (0.50,-0.19) -- (0.59,-0.13) -- (0.69,-0.05) -- (0.78,0.04) -- (0.88,0.14) -- (0.98,0.26) -- (1.07,0.40) -- (1.17,0.55) -- (1.17,0.55) -- (1.13,0.61) -- (1.09,0.67) -- (1.04,0.73) -- (1.00,0.78) -- (0.96,0.84) -- (0.92,0.90) -- (0.88,0.96) -- (0.84,1.02) -- (0.80,1.08) -- (0.75,1.14) -- (0.75,1.14) -- (0.66,0.99) -- (0.56,0.85) -- (0.47,0.73) -- (0.37,0.62) -- (0.27,0.53) -- (0.18,0.46) -- (0.08,0.39) -- (-0.01,0.34) -- (-0.11,0.31) -- (-0.21,0.29) -- (-0.30,0.29) -- (-0.40,0.30) -- (-0.50,0.33) -- (-0.59,0.37) -- (-0.69,0.42) -- (-0.78,0.49) -- (-0.88,0.57) -- (-0.98,0.67) -- (-1.07,0.78) -- (-1.17,0.91) -- (-1.17,0.91) -- (-1.13,0.85) -- (-1.09,0.79) -- (-1.04,0.74) -- (-1.00,0.68) -- (-0.96,0.62) -- (-0.92,0.56) -- (-0.88,0.50) -- (-0.84,0.44) -- (-0.80,0.38) -- (-0.75,0.33) -- cycle;
\fill[gray!55] (-0.23,-0.13) -- (-0.19,-0.15) -- (-0.16,-0.16) -- (-0.12,-0.17) -- (-0.08,-0.18) -- (-0.04,-0.18) -- (0.01,-0.19) -- (0.06,-0.19) -- (0.12,-0.18) -- (0.17,-0.17) -- (0.22,-0.16) -- (0.27,-0.14) -- (0.31,-0.12) -- (0.35,-0.10) -- (0.39,-0.08) -- (0.43,-0.05) -- (0.46,-0.03) -- (0.43,-0.02) -- (0.39,-0.01) -- (0.36,0.01) -- (0.32,0.02) -- (0.28,0.03) -- (0.25,0.05) -- (0.21,0.06) -- (0.17,0.08) -- (0.13,0.09) -- (0.10,0.11) -- (0.06,0.12) -- (0.02,0.14) -- (-0.02,0.16) -- (-0.06,0.18) -- (-0.09,0.20) -- (-0.13,0.22) -- (-0.14,0.19) -- (-0.14,0.16) -- (-0.15,0.13) -- (-0.16,0.11) -- (-0.17,0.08) -- (-0.17,0.06) -- (-0.18,0.04) -- (-0.19,0.01) -- (-0.20,-0.01) -- (-0.20,-0.03) -- (-0.21,-0.05) -- (-0.21,-0.06) -- (-0.22,-0.08) -- (-0.22,-0.10) -- (-0.22,-0.12) -- cycle;
\draw[black, line width=0.6pt, densely dashed] (-0.23,-0.13) -- (0.46,-0.03) -- (-0.13,0.22) -- cycle;
\fill (-0.227,-0.133) circle (0.9pt);
\fill (0.462,-0.029) circle (0.9pt);
\fill (-0.132,0.218) circle (0.9pt);
\fill[gray!8] (2.00,0.33) -- (2.09,0.20) -- (2.19,0.08) -- (2.28,-0.01) -- (2.38,-0.10) -- (2.48,-0.17) -- (2.57,-0.22) -- (2.67,-0.26) -- (2.76,-0.29) -- (2.86,-0.30) -- (2.96,-0.29) -- (3.05,-0.27) -- (3.15,-0.24) -- (3.25,-0.19) -- (3.34,-0.13) -- (3.44,-0.05) -- (3.53,0.04) -- (3.63,0.14) -- (3.73,0.26) -- (3.82,0.40) -- (3.92,0.55) -- (3.92,0.55) -- (3.88,0.61) -- (3.84,0.67) -- (3.79,0.73) -- (3.75,0.78) -- (3.71,0.84) -- (3.67,0.90) -- (3.63,0.96) -- (3.59,1.02) -- (3.55,1.08) -- (3.50,1.14) -- (3.50,1.14) -- (3.41,0.99) -- (3.31,0.85) -- (3.22,0.73) -- (3.12,0.62) -- (3.02,0.53) -- (2.93,0.46) -- (2.83,0.39) -- (2.74,0.34) -- (2.64,0.31) -- (2.54,0.29) -- (2.45,0.29) -- (2.35,0.30) -- (2.25,0.33) -- (2.16,0.37) -- (2.06,0.42) -- (1.97,0.49) -- (1.87,0.57) -- (1.77,0.67) -- (1.68,0.78) -- (1.58,0.91) -- (1.58,0.91) -- (1.62,0.85) -- (1.66,0.79) -- (1.71,0.74) -- (1.75,0.68) -- (1.79,0.62) -- (1.83,0.56) -- (1.87,0.50) -- (1.91,0.44) -- (1.95,0.38) -- (2.00,0.33) -- cycle;
\draw[gray!45, line width=0.3pt] (2.00,0.33) -- (2.09,0.20) -- (2.19,0.08) -- (2.28,-0.01) -- (2.38,-0.10) -- (2.48,-0.17) -- (2.57,-0.22) -- (2.67,-0.26) -- (2.76,-0.29) -- (2.86,-0.30) -- (2.96,-0.29) -- (3.05,-0.27) -- (3.15,-0.24) -- (3.25,-0.19) -- (3.34,-0.13) -- (3.44,-0.05) -- (3.53,0.04) -- (3.63,0.14) -- (3.73,0.26) -- (3.82,0.40) -- (3.92,0.55);
\draw[gray!45, line width=0.3pt] (1.93,0.42) -- (2.02,0.30) -- (2.12,0.18) -- (2.21,0.08) -- (2.31,0.00) -- (2.41,-0.07) -- (2.50,-0.12) -- (2.60,-0.16) -- (2.70,-0.19) -- (2.79,-0.20) -- (2.89,-0.20) -- (2.98,-0.18) -- (3.08,-0.14) -- (3.18,-0.10) -- (3.27,-0.03) -- (3.37,0.04) -- (3.46,0.13) -- (3.56,0.24) -- (3.66,0.36) -- (3.75,0.50) -- (3.85,0.65);
\draw[gray!45, line width=0.3pt] (1.86,0.52) -- (1.95,0.39) -- (2.05,0.28) -- (2.15,0.18) -- (2.24,0.10) -- (2.34,0.03) -- (2.43,-0.03) -- (2.53,-0.07) -- (2.63,-0.09) -- (2.72,-0.10) -- (2.82,-0.10) -- (2.92,-0.08) -- (3.01,-0.05) -- (3.11,0.00) -- (3.20,0.06) -- (3.30,0.14) -- (3.40,0.23) -- (3.49,0.34) -- (3.59,0.46) -- (3.68,0.60) -- (3.78,0.75);
\draw[gray!45, line width=0.3pt] (1.79,0.62) -- (1.88,0.49) -- (1.98,0.38) -- (2.08,0.28) -- (2.17,0.20) -- (2.27,0.13) -- (2.37,0.07) -- (2.46,0.03) -- (2.56,0.01) -- (2.65,-0.00) -- (2.75,0.00) -- (2.85,0.02) -- (2.94,0.05) -- (3.04,0.10) -- (3.13,0.16) -- (3.23,0.24) -- (3.33,0.33) -- (3.42,0.44) -- (3.52,0.56) -- (3.62,0.69) -- (3.71,0.84);
\draw[gray!45, line width=0.3pt] (1.72,0.72) -- (1.82,0.59) -- (1.91,0.48) -- (2.01,0.38) -- (2.10,0.29) -- (2.20,0.22) -- (2.30,0.17) -- (2.39,0.13) -- (2.49,0.10) -- (2.58,0.09) -- (2.68,0.10) -- (2.78,0.12) -- (2.87,0.15) -- (2.97,0.20) -- (3.07,0.26) -- (3.16,0.34) -- (3.26,0.43) -- (3.35,0.53) -- (3.45,0.66) -- (3.55,0.79) -- (3.64,0.94);
\draw[gray!45, line width=0.3pt] (1.65,0.81) -- (1.75,0.69) -- (1.84,0.57) -- (1.94,0.48) -- (2.04,0.39) -- (2.13,0.32) -- (2.23,0.27) -- (2.32,0.23) -- (2.42,0.20) -- (2.52,0.19) -- (2.61,0.20) -- (2.71,0.21) -- (2.80,0.25) -- (2.90,0.29) -- (3.00,0.36) -- (3.09,0.43) -- (3.19,0.53) -- (3.29,0.63) -- (3.38,0.75) -- (3.48,0.89) -- (3.57,1.04);
\draw[gray!45, line width=0.3pt] (1.58,0.91) -- (1.68,0.78) -- (1.77,0.67) -- (1.87,0.57) -- (1.97,0.49) -- (2.06,0.42) -- (2.16,0.37) -- (2.25,0.33) -- (2.35,0.30) -- (2.45,0.29) -- (2.54,0.29) -- (2.64,0.31) -- (2.74,0.34) -- (2.83,0.39) -- (2.93,0.46) -- (3.02,0.53) -- (3.12,0.62) -- (3.22,0.73) -- (3.31,0.85) -- (3.41,0.99) -- (3.50,1.14);
\draw[gray!45, line width=0.3pt] (2.00,0.33) -- (1.95,0.38) -- (1.91,0.44) -- (1.87,0.50) -- (1.83,0.56) -- (1.79,0.62) -- (1.75,0.68) -- (1.71,0.74) -- (1.66,0.79) -- (1.62,0.85) -- (1.58,0.91);
\draw[gray!45, line width=0.3pt] (2.24,0.03) -- (2.19,0.09) -- (2.15,0.15) -- (2.11,0.21) -- (2.07,0.27) -- (2.03,0.33) -- (1.99,0.39) -- (1.95,0.44) -- (1.91,0.50) -- (1.86,0.56) -- (1.82,0.62);
\draw[gray!45, line width=0.3pt] (2.48,-0.17) -- (2.43,-0.11) -- (2.39,-0.05) -- (2.35,0.01) -- (2.31,0.07) -- (2.27,0.13) -- (2.23,0.19) -- (2.19,0.24) -- (2.15,0.30) -- (2.10,0.36) -- (2.06,0.42);
\draw[gray!45, line width=0.3pt] (2.72,-0.28) -- (2.68,-0.22) -- (2.63,-0.16) -- (2.59,-0.10) -- (2.55,-0.04) -- (2.51,0.02) -- (2.47,0.08) -- (2.43,0.13) -- (2.39,0.19) -- (2.34,0.25) -- (2.30,0.31);
\draw[gray!45, line width=0.3pt] (2.96,-0.29) -- (2.92,-0.23) -- (2.87,-0.18) -- (2.83,-0.12) -- (2.79,-0.06) -- (2.75,0.00) -- (2.71,0.06) -- (2.67,0.12) -- (2.63,0.18) -- (2.58,0.23) -- (2.54,0.29);
\draw[gray!45, line width=0.3pt] (3.20,-0.22) -- (3.16,-0.16) -- (3.11,-0.10) -- (3.07,-0.04) -- (3.03,0.02) -- (2.99,0.07) -- (2.95,0.13) -- (2.91,0.19) -- (2.87,0.25) -- (2.82,0.31) -- (2.78,0.37);
\draw[gray!45, line width=0.3pt] (3.44,-0.05) -- (3.40,0.00) -- (3.35,0.06) -- (3.31,0.12) -- (3.27,0.18) -- (3.23,0.24) -- (3.19,0.30) -- (3.15,0.36) -- (3.11,0.41) -- (3.07,0.47) -- (3.02,0.53);
\draw[gray!45, line width=0.3pt] (3.68,0.20) -- (3.64,0.26) -- (3.59,0.32) -- (3.55,0.38) -- (3.51,0.44) -- (3.47,0.50) -- (3.43,0.55) -- (3.39,0.61) -- (3.35,0.67) -- (3.31,0.73) -- (3.26,0.79);
\draw[gray!45, line width=0.3pt] (3.92,0.55) -- (3.88,0.61) -- (3.84,0.67) -- (3.79,0.73) -- (3.75,0.78) -- (3.71,0.84) -- (3.67,0.90) -- (3.63,0.96) -- (3.59,1.02) -- (3.55,1.08) -- (3.50,1.14);
\draw[gray!70, line width=0.5pt] (2.00,0.33) -- (2.09,0.20) -- (2.19,0.08) -- (2.28,-0.01) -- (2.38,-0.10) -- (2.48,-0.17) -- (2.57,-0.22) -- (2.67,-0.26) -- (2.76,-0.29) -- (2.86,-0.30) -- (2.96,-0.29) -- (3.05,-0.27) -- (3.15,-0.24) -- (3.25,-0.19) -- (3.34,-0.13) -- (3.44,-0.05) -- (3.53,0.04) -- (3.63,0.14) -- (3.73,0.26) -- (3.82,0.40) -- (3.92,0.55) -- (3.92,0.55) -- (3.88,0.61) -- (3.84,0.67) -- (3.79,0.73) -- (3.75,0.78) -- (3.71,0.84) -- (3.67,0.90) -- (3.63,0.96) -- (3.59,1.02) -- (3.55,1.08) -- (3.50,1.14) -- (3.50,1.14) -- (3.41,0.99) -- (3.31,0.85) -- (3.22,0.73) -- (3.12,0.62) -- (3.02,0.53) -- (2.93,0.46) -- (2.83,0.39) -- (2.74,0.34) -- (2.64,0.31) -- (2.54,0.29) -- (2.45,0.29) -- (2.35,0.30) -- (2.25,0.33) -- (2.16,0.37) -- (2.06,0.42) -- (1.97,0.49) -- (1.87,0.57) -- (1.77,0.67) -- (1.68,0.78) -- (1.58,0.91) -- (1.58,0.91) -- (1.62,0.85) -- (1.66,0.79) -- (1.71,0.74) -- (1.75,0.68) -- (1.79,0.62) -- (1.83,0.56) -- (1.87,0.50) -- (1.91,0.44) -- (1.95,0.38) -- (2.00,0.33) -- cycle;
\fill[gray!55] (1.98,0.38) -- (2.00,0.36) -- (2.02,0.33) -- (2.05,0.31) -- (2.07,0.29) -- (2.09,0.27) -- (2.12,0.25) -- (2.14,0.23) -- (2.17,0.21) -- (2.20,0.18) -- (2.23,0.17) -- (2.26,0.15) -- (2.29,0.13) -- (2.32,0.11) -- (2.36,0.09) -- (2.39,0.08) -- (2.43,0.07) -- (2.46,0.06) -- (2.50,0.05) -- (2.54,0.04) -- (2.58,0.04) -- (2.62,0.04) -- (2.65,0.04) -- (2.69,0.04) -- (2.72,0.05) -- (2.75,0.05) -- (2.79,0.05) -- (2.83,0.06) -- (2.87,0.07) -- (2.91,0.08) -- (2.95,0.10) -- (2.99,0.12) -- (3.03,0.14) -- (3.07,0.16) -- (3.11,0.18) -- (3.15,0.20) -- (3.18,0.23) -- (3.22,0.25) -- (3.25,0.28) -- (3.28,0.31) -- (3.31,0.33) -- (3.34,0.36) -- (3.37,0.39) -- (3.40,0.41) -- (3.42,0.44) -- (3.45,0.47) -- (3.48,0.50) -- (3.50,0.52) -- (3.52,0.55) -- (3.50,0.52) -- (3.47,0.48) -- (3.44,0.45) -- (3.41,0.41) -- (3.38,0.37) -- (3.34,0.33) -- (3.30,0.29) -- (3.25,0.25) -- (3.19,0.20) -- (3.11,0.14) -- (3.00,0.07) -- (2.75,-0.00) -- (2.51,0.01) -- (2.39,0.06) -- (2.32,0.10) -- (2.26,0.13) -- (2.21,0.17) -- (2.17,0.20) -- (2.13,0.23) -- (2.09,0.26) -- (2.06,0.29) -- (2.03,0.32) -- (2.01,0.35) -- cycle;
\draw[black, line width=0.6pt, densely dashed, fill=gray!25, fill opacity=0.55] (1.98,0.38) -- (2.72,0.05) -- (3.52,0.55) -- cycle;
\draw[-{Latex[length=2.4pt,width=1.8pt]}, black!80, line width=0.45pt] (2.749,0.422) -- (2.749,0.001);
\draw[-{Latex[length=2.4pt,width=1.8pt]}, black!80, line width=0.45pt] (2.738,0.297) -- (2.738,0.016);
\draw[-{Latex[length=2.4pt,width=1.8pt]}, black!80, line width=0.45pt] (2.437,0.346) -- (2.273,0.129);
\draw[-{Latex[length=2.4pt,width=1.8pt]}, black!80, line width=0.45pt] (3.054,0.415) -- (3.218,0.235);
\fill (1.981,0.378) circle (0.9pt);
\fill (2.717,0.047) circle (0.9pt);
\fill (3.525,0.550) circle (0.9pt);
\node at (2.47,0.35) {$\tau$};
\draw[gray!60, line width=0.3pt] (3.17,0.20) -- (3.42,-0.08) node[anchor=north west, inner sep=1pt, text=black] {$\pi_\N(\tau)$};
\node at (0,-0.42) {(a) shape-regular: $J_{\pi_\N|_\tau}\approx1$};
\node at (2.75,-0.42) {(b) thin: $J_{\pi_\N|_\tau}\ll1$};
\end{tikzpicture}
\caption{Projection of a triangle $\tau$ (dashed) onto $\N$ (dark gray): (a) a shape-regular triangle keeps its area; (b) a thin triangle may collapse.}
\label{fig:thin_triangle}
\end{figure}

\begin{lemma}\label{lem:thin_triangle}
	Let $\N$ be a smooth surface and $\tau$ be a triangle with vertices on $\N$ whose diameter is $k_h$. If $k_h$ is sufficiently small, then
	\begin{equation} \label{eq:thin_triangle}
	\sup_{y\in\tau}\operatorname{dist}(y,\N)\leq Ck_h^2,
	\qquad
	0\leq J_{\pi_\N|_\tau}(y)\leq 1+Ck_h^2
	\quad\text{for a.e. }y\in\tau,
	\end{equation}
	where the constant $C$ depends only on $\N$.
\end{lemma}

\begin{proof}
	Let $v_0,v_1,v_2\in\N$ be the vertices of $\tau$. The barycentric coordinates~\eqref{eq:bary_coord} give the identities:
	\begin{equation} \label{eq:thin_triangle_1}
	y=\sum_\ell\lambda_\ell v_\ell,
	\qquad \lambda_\ell\geq0,
	\qquad \sum_\ell\lambda_\ell=1.
	\end{equation}
	For sufficiently small $k_h$, $\tau$ lies in a fixed tubular neighborhood of $\N$. Since $v_0,v_1,v_2\in\N$, the signed distance \eqref{eq:patch_projection} on $\N$ is $t_\N(v_i) = 0$ for $i = 0,1,2$, and Taylor's expansion at $y$ gives
	\[
	0 = t_\N(v_i)
	= t_\N(y) + \nabla t_\N (y) \cdot (v_i-y)+R_i,
	\qquad |R_i| \leq Ck_h^2.
	\]
	Multiplying by $\lambda_i$ and summing gives
	\[
	0 = \sum_{i} \lambda_i t_\N(y) +  \sum_{i} \lambda_i \Big( \nabla t_\N (y) \cdot (v_i-y) \Big) +  \sum_{i} \lambda_i R_i.
	\]
	By \eqref{eq:thin_triangle_1}, the coefficient of the first term is 1 and the second term vanishes. Hence
	\begin{equation} \label{eq:thin_triangle_2}
	\operatorname{dist}(y,\N)=|t_\N(y)|=\Big|\sum_i\lambda_i R_i\Big|\leq Ck_h^2.
	\end{equation}

	For the Jacobian, let $p=\pi_\N(y) \in \N$. Since $\pi_\N$ fixes $\N$ and is constant along the normal through $p$, its differential at $p$ is the orthogonal projection onto $T_p\N$, i.e., $\mathrm D\pi_\N(p)=\Pi_p$. For $y$ in the tubular neighborhood, the mean value theorem and \eqref{eq:thin_triangle_2} give
	\[
	\|\mathrm D\pi_\N(y)-\Pi_p\|
	=\|\mathrm D\pi_\N(y)-\mathrm D\pi_\N(p)\|
	\leq C|y-p|\leq Ck_h^2.
	\]
	Since $\|\Pi_p\|=1$, it follows that
	\begin{equation} \label{eq:thin_triangle_3}
	\|\mathrm D\pi_\N(y)\|\leq1+Ck_h^2.
	\end{equation}
	Consequently, choosing $e_1,e_2$ to be an orthonormal basis of the plane of $\tau$, we obtain by \eqref{eq:thin_triangle_3}
	\[
	0\leq J_{\pi_\N|_\tau}(y)
	=|\mathrm D\pi_\N(y)e_1\times\mathrm D\pi_\N(y)e_2|
	\leq\|\mathrm D\pi_\N(y)\|^2
	\leq(1+Ck_h^2)^2\leq1+Ck_h^2.
	\]
\end{proof}

\subsection{Boundary correction}
\label{sec:open}
If $\partial\M = \varnothing$, then $\pi_\M(\Mh) = \M$ and $\pi_\N(\Nh) = \N$. Hence, we can transform the simplicial map into a homeomorphism by
\begin{equation*}
	\pi_\N \circ f_h \circ \pi_\M^{-1} :\M \to \N.
\end{equation*}
However, if $\partial\M \neq \varnothing$, $\pi_\M(\Mh) \neq \M$, so we require the additional correction maps $r_h$ and $s_h$ such that
\[
\begin{array}{ccccc}
	\Mh & \xrightarrow{\; \pi_\M \; } & \pi_\M(\Mh)  & \xrightarrow{\; r_h\; }  & \M  \\[2pt]
	{\scriptstyle f_h}\big\downarrow\ & & &&\\[2pt]
	\Nh & \xrightarrow{\; \pi_\N  \; } & \pi_\N(\Nh)  &   \xrightarrow{\; s_h\;}  & \N
\end{array}
\]
Hence, we let $p_h = r_h \circ \pi_\M$ and $q_h = s_h \circ \pi_\N$, and define the lifting operator $\mathrm{L}_h$ by
\begin{equation*}
	\mathrm{L}_h f_h
	:= q_h \circ f_h \circ p_h^{-1} \in\mathrm{Homeo}^+(\M,\N).
\end{equation*}

The projections $\pi_\M$ and $\pi_\N$ have two-sided and one-sided control of the Jacobian by Lemma~\ref{lem:area_change} and Lemma~\ref{lem:thin_triangle}, respectively. The correction maps $r_h$ and $s_h$ should keep the same order of error (Figure~\ref{fig:boundary_correction}). To do this, we first provide the following lemma.

\begin{figure}[tb]
\centering
\begin{tikzpicture}[font=\footnotesize]
\begin{scope}[x=1.3cm, y=1.3cm]
\fill[gray!12] (-1.60,-0.16) -- (-1.55,-0.13) -- (-1.49,-0.09) -- (-1.44,-0.06) -- (-1.39,-0.04) -- (-1.33,-0.02) -- (-1.28,-0.00) -- (-1.23,0.01) -- (-1.17,0.01) -- (-1.12,0.01) -- (-1.07,-0.00) -- (-1.01,-0.02) -- (-0.96,-0.04) -- (-0.91,-0.06) -- (-0.85,-0.09) -- (-0.80,-0.11) -- (-0.75,-0.14) -- (-0.69,-0.17) -- (-0.64,-0.19) -- (-0.59,-0.21) -- (-0.53,-0.22) -- (-0.48,-0.23) -- (-0.43,-0.23) -- (-0.37,-0.22) -- (-0.32,-0.20) -- (-0.27,-0.18) -- (-0.21,-0.15) -- (-0.16,-0.12) -- (-0.11,-0.08) -- (-0.05,-0.04) -- (0.00,0.00) -- (0.05,0.04) -- (0.11,0.08) -- (0.16,0.12) -- (0.21,0.15) -- (0.27,0.18) -- (0.32,0.20) -- (0.37,0.22) -- (0.43,0.23) -- (0.48,0.23) -- (0.53,0.22) -- (0.59,0.21) -- (0.64,0.19) -- (0.69,0.17) -- (0.75,0.14) -- (0.80,0.11) -- (0.86,0.16) -- (0.91,0.20) -- (0.97,0.23) -- (1.03,0.26) -- (1.09,0.29) -- (1.14,0.30) -- (1.20,0.31) -- (1.26,0.31) -- (1.31,0.30) -- (1.37,0.28) -- (1.43,0.26) -- (1.49,0.23) -- (1.54,0.20) -- (1.60,0.16) -- (1.6,-0.95) -- (-1.6,-0.95) -- cycle;
\draw[gray!65, densely dotted, line width=0.5pt] (-1.60,0.46) -- (-1.50,0.46) -- (-1.40,0.46) -- (-1.30,0.47) -- (-1.20,0.47) -- (-1.10,0.48) -- (-1.00,0.49) -- (-0.90,0.49) -- (-0.80,0.51) -- (-0.70,0.52) -- (-0.60,0.53) -- (-0.50,0.54) -- (-0.40,0.56) -- (-0.30,0.57) -- (-0.20,0.59) -- (-0.10,0.60) -- (0.00,0.62) -- (0.10,0.64) -- (0.20,0.65) -- (0.30,0.67) -- (0.40,0.68) -- (0.50,0.70) -- (0.60,0.71) -- (0.70,0.72) -- (0.80,0.73) -- (0.90,0.75) -- (1.00,0.75) -- (1.10,0.76) -- (1.20,0.77) -- (1.30,0.77) -- (1.40,0.78) -- (1.50,0.78) -- (1.60,0.78);
\draw[gray!65, densely dotted, line width=0.5pt] (-1.60,-0.78) -- (-1.50,-0.78) -- (-1.40,-0.78) -- (-1.30,-0.77) -- (-1.20,-0.77) -- (-1.10,-0.76) -- (-1.00,-0.75) -- (-0.90,-0.75) -- (-0.80,-0.73) -- (-0.70,-0.72) -- (-0.60,-0.71) -- (-0.50,-0.70) -- (-0.40,-0.68) -- (-0.30,-0.67) -- (-0.20,-0.65) -- (-0.10,-0.64) -- (0.00,-0.62) -- (0.10,-0.60) -- (0.20,-0.59) -- (0.30,-0.57) -- (0.40,-0.56) -- (0.50,-0.54) -- (0.60,-0.53) -- (0.70,-0.52) -- (0.80,-0.51) -- (0.90,-0.49) -- (1.00,-0.49) -- (1.10,-0.48) -- (1.20,-0.47) -- (1.30,-0.47) -- (1.40,-0.46) -- (1.50,-0.46) -- (1.60,-0.46);
\draw[black, densely dashed, line width=0.6pt] (-1.60,-0.16) -- (-1.50,-0.16) -- (-1.40,-0.16) -- (-1.30,-0.15) -- (-1.20,-0.15) -- (-1.10,-0.14) -- (-1.00,-0.13) -- (-0.90,-0.13) -- (-0.80,-0.11) -- (-0.70,-0.10) -- (-0.60,-0.09) -- (-0.50,-0.08) -- (-0.40,-0.06) -- (-0.30,-0.05) -- (-0.20,-0.03) -- (-0.10,-0.02) -- (0.00,0.00) -- (0.10,0.02) -- (0.20,0.03) -- (0.30,0.05) -- (0.40,0.06) -- (0.50,0.08) -- (0.60,0.09) -- (0.70,0.10) -- (0.80,0.11) -- (0.90,0.13) -- (1.00,0.13) -- (1.10,0.14) -- (1.20,0.15) -- (1.30,0.15) -- (1.40,0.16) -- (1.50,0.16) -- (1.60,0.16);
\draw[black, line width=0.9pt] (-1.60,-0.16) -- (-1.55,-0.13) -- (-1.49,-0.09) -- (-1.44,-0.06) -- (-1.39,-0.04) -- (-1.33,-0.02) -- (-1.28,-0.00) -- (-1.23,0.01) -- (-1.17,0.01) -- (-1.12,0.01) -- (-1.07,-0.00) -- (-1.01,-0.02) -- (-0.96,-0.04) -- (-0.91,-0.06) -- (-0.85,-0.09) -- (-0.80,-0.11) -- (-0.75,-0.14) -- (-0.69,-0.17) -- (-0.64,-0.19) -- (-0.59,-0.21) -- (-0.53,-0.22) -- (-0.48,-0.23) -- (-0.43,-0.23) -- (-0.37,-0.22) -- (-0.32,-0.20) -- (-0.27,-0.18) -- (-0.21,-0.15) -- (-0.16,-0.12) -- (-0.11,-0.08) -- (-0.05,-0.04) -- (0.00,0.00) -- (0.05,0.04) -- (0.11,0.08) -- (0.16,0.12) -- (0.21,0.15) -- (0.27,0.18) -- (0.32,0.20) -- (0.37,0.22) -- (0.43,0.23) -- (0.48,0.23) -- (0.53,0.22) -- (0.59,0.21) -- (0.64,0.19) -- (0.69,0.17) -- (0.75,0.14) -- (0.80,0.11) -- (0.86,0.16) -- (0.91,0.20) -- (0.97,0.23) -- (1.03,0.26) -- (1.09,0.29) -- (1.14,0.30) -- (1.20,0.31) -- (1.26,0.31) -- (1.31,0.30) -- (1.37,0.28) -- (1.43,0.26) -- (1.49,0.23) -- (1.54,0.20) -- (1.60,0.16);
\fill (-1.60,-0.16) circle (1.3pt);
\fill (-0.80,-0.11) circle (1.3pt);
\fill (0.00,0.00) circle (1.3pt);
\fill (0.80,0.11) circle (1.3pt);
\fill (1.60,0.16) circle (1.3pt);
\draw[decorate, decoration={brace, amplitude=3pt}, line width=0.4pt] (1.72,0.78) -- (1.72,-0.46) node[midway, right=3pt] {$U$};
\draw[gray!60, line width=0.3pt] (-1.2,0.01) -- (-1.45,0.95) node[anchor=south, inner sep=1pt, text=black] {$\pi_\M(\partial\Mh)$};
\draw[gray!60, line width=0.3pt] (0.4,0.06) -- (0.55,0.95) node[anchor=south, inner sep=1pt, text=black] {$\partial\M$};
\node at (0.0,-0.72) {$\pi_\M(\Mh)$};
\node at (0,-1.3) {(a) in $\M^e$};
\end{scope}
\draw[-{Latex[length=4pt]}, line width=0.5pt] (3.05,0.1) -- (3.85,0.1) node[midway, above] {$\Psi^{-1}$};
\begin{scope}[xshift=6.45cm, x=1.3cm, y=1.3cm]
\fill[gray!12] (-1.60,0.00) -- (-1.55,0.04) -- (-1.49,0.07) -- (-1.44,0.11) -- (-1.39,0.13) -- (-1.33,0.16) -- (-1.28,0.17) -- (-1.23,0.18) -- (-1.17,0.18) -- (-1.12,0.17) -- (-1.07,0.16) -- (-1.01,0.13) -- (-0.96,0.11) -- (-0.91,0.07) -- (-0.85,0.04) -- (-0.80,-0.00) -- (-0.75,-0.04) -- (-0.69,-0.07) -- (-0.64,-0.11) -- (-0.59,-0.13) -- (-0.53,-0.16) -- (-0.48,-0.17) -- (-0.43,-0.18) -- (-0.37,-0.18) -- (-0.32,-0.17) -- (-0.27,-0.16) -- (-0.21,-0.13) -- (-0.16,-0.11) -- (-0.11,-0.07) -- (-0.05,-0.04) -- (0.00,0.00) -- (0.05,0.04) -- (0.11,0.07) -- (0.16,0.11) -- (0.21,0.13) -- (0.27,0.16) -- (0.32,0.17) -- (0.37,0.18) -- (0.43,0.18) -- (0.48,0.17) -- (0.53,0.16) -- (0.59,0.13) -- (0.64,0.11) -- (0.69,0.07) -- (0.75,0.04) -- (0.80,0.00) -- (0.86,0.04) -- (0.91,0.08) -- (0.97,0.11) -- (1.03,0.14) -- (1.09,0.16) -- (1.14,0.18) -- (1.20,0.18) -- (1.26,0.18) -- (1.31,0.16) -- (1.37,0.14) -- (1.43,0.11) -- (1.49,0.08) -- (1.54,0.04) -- (1.60,0.00) -- (1.6,0.62) -- (-1.6,0.62) -- cycle;
\draw[gray!65, densely dotted, line width=0.5pt] (-1.6,0.62) -- (1.6,0.62);
\draw[gray!65, densely dotted, line width=0.5pt] (-1.6,-0.62) -- (1.6,-0.62);
\draw[black, densely dashed, line width=0.6pt] (-1.6,0) -- (1.6,0);
\draw[black, line width=0.9pt] (-1.60,0.00) -- (-1.55,0.04) -- (-1.49,0.07) -- (-1.44,0.11) -- (-1.39,0.13) -- (-1.33,0.16) -- (-1.28,0.17) -- (-1.23,0.18) -- (-1.17,0.18) -- (-1.12,0.17) -- (-1.07,0.16) -- (-1.01,0.13) -- (-0.96,0.11) -- (-0.91,0.07) -- (-0.85,0.04) -- (-0.80,-0.00) -- (-0.75,-0.04) -- (-0.69,-0.07) -- (-0.64,-0.11) -- (-0.59,-0.13) -- (-0.53,-0.16) -- (-0.48,-0.17) -- (-0.43,-0.18) -- (-0.37,-0.18) -- (-0.32,-0.17) -- (-0.27,-0.16) -- (-0.21,-0.13) -- (-0.16,-0.11) -- (-0.11,-0.07) -- (-0.05,-0.04) -- (0.00,0.00) -- (0.05,0.04) -- (0.11,0.07) -- (0.16,0.11) -- (0.21,0.13) -- (0.27,0.16) -- (0.32,0.17) -- (0.37,0.18) -- (0.43,0.18) -- (0.48,0.17) -- (0.53,0.16) -- (0.59,0.13) -- (0.64,0.11) -- (0.69,0.07) -- (0.75,0.04) -- (0.80,0.00) -- (0.86,0.04) -- (0.91,0.08) -- (0.97,0.11) -- (1.03,0.14) -- (1.09,0.16) -- (1.14,0.18) -- (1.20,0.18) -- (1.26,0.18) -- (1.31,0.16) -- (1.37,0.14) -- (1.43,0.11) -- (1.49,0.08) -- (1.54,0.04) -- (1.60,0.00);
\fill (-1.60,0) circle (1.3pt);
\fill (-0.80,0) circle (1.3pt);
\fill (0.00,0) circle (1.3pt);
\fill (0.80,0) circle (1.3pt);
\fill (1.60,0) circle (1.3pt);
\draw[-{Latex[length=3.5pt,width=2.6pt]}, line width=0.5pt] (-1.20,0.17) -- (-1.20,0.03);
\draw[-{Latex[length=3.5pt,width=2.6pt]}, line width=0.5pt] (-0.40,-0.17) -- (-0.40,-0.03);
\draw[-{Latex[length=3.5pt,width=2.6pt]}, line width=0.5pt] (0.40,0.17) -- (0.40,0.03);
\draw[-{Latex[length=3.5pt,width=2.6pt]}, line width=0.5pt] (1.20,0.17) -- (1.20,0.03);
\draw[-{Latex[length=3.5pt]}, line width=0.4pt] (-1.75,-0.72) -- (-1.75,0.82) node[anchor=south] {$t$};
\draw[-{Latex[length=3.5pt]}, line width=0.4pt] (-1.75,-0.72) -- (1.8,-0.72) node[anchor=west] {$s$};
\node[anchor=east, text=gray!50!black] at (-1.75,0.62) {$\delta$};
\node[anchor=east, text=gray!50!black] at (-1.75,-0.62) {$-\delta$};
\node[anchor=west] at (1.62,0) {$t=0$};
\draw[gray!60, line width=0.3pt] (-0.25,-0.15) -- (-0.05,-0.4) node[anchor=west, inner sep=1pt, text=black] {$t=\beta_h(s)$};
\node at (0,-1.3) {(b) in coordinates $(s,t)$};
\end{scope}
\end{tikzpicture}
\caption{The boundary correction: (a) the lifted boundary (solid) near $\partial\M$ (dashed) in $U$; (b) $P_h$ pushes $t=\beta_h(s)$ to $t=0$.}
\label{fig:boundary_correction}
\end{figure}
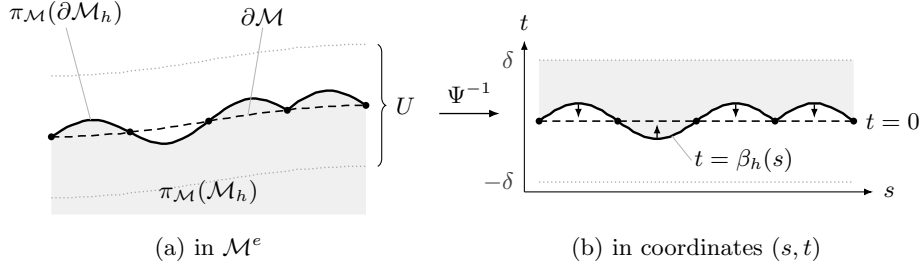

\begin{lemma}\label{lem:ass_4_guarantee_1}
	Let $\M$ be a compact smooth surface with $\partial \M \neq \varnothing$ and let $\Mh$ be the triangular mesh approximating $\M$ satisfying Assumption~\ref{ass:source}. For sufficiently small $h$, $\partial\M$ and $\pi_{\M}(\partial\Mh)$ differ by $\O(h^2)$ in position and $\O(h)$ in tangential direction.
\end{lemma}

\begin{proof}
	We first compare a boundary edge of $\Mh$ with the corresponding smooth arc of $\partial\M$. We choose an edge between two adjacent boundary vertices $v_0$ and $v_1$ and parameterize the
	corresponding arc of $\partial\M$ by arclength
	\[
	\gamma:[0,\ell]\to\partial\M, \qquad
	|\gamma'(u)|=1, \qquad \gamma(0)=v_0, \qquad \gamma(\ell)=v_1,
	\]
	where $\ell\leq Ch$ for some $C>0$. We write the boundary edge as the chord
	\begin{equation}\label{eq:ass_4_guarantee_1}
		c(u)=\gamma(0)+\frac{u}{\ell}\big(\gamma(\ell)-\gamma(0)\big),
		\qquad 0\leq u\leq\ell.
	\end{equation}
	Taylor's expansion of $\gamma$ at $u=0$ gives
	\begin{equation}\label{eq:ass_4_guarantee_2}
		\gamma(u)=\gamma(0)+u\,\gamma'(0)+R(u),
		\qquad |R(u)|\leq Cu^2.
	\end{equation}
	Evaluating \eqref{eq:ass_4_guarantee_2} at $u=\ell$ and substituting into 	\eqref{eq:ass_4_guarantee_1}, we obtain
	\begin{align*}
		c(u)
		=\gamma(0)+\frac{u}{\ell}\big(\gamma(\ell)-\gamma(0)\big)
		=\gamma(0)+u\,\gamma'(0)+\frac{u}{\ell}R(\ell).
	\end{align*}
	Hence, the position difference is
	\begin{equation}\label{eq:boundary_estimate_3}
		|c(u)-\gamma(u)|
		=\left|\frac{u}{\ell}R(\ell)-R(u)\right|
		\leq\frac{u}{\ell}|R(\ell)|+|R(u)|
		\leq Cu\ell+Cu^2\leq Ch^2
	\end{equation}
	for all $0\leq u\leq\ell$. In addition, taking derivatives of \eqref{eq:ass_4_guarantee_1} gives
	\[
	|c'(u) - \gamma'(0)| = |R(\ell) / \ell| \leq C\ell.
	\]
	Since $\gamma''$ is uniformly bounded, we obtain
	\begin{equation}\label{eq:ass_4_guarantee_4}
	|c'(u)-\gamma'(u)| \leq |c'(u) - \gamma'(0)| + |\gamma'(u)-\gamma'(0)| \leq C \ell \leq Ch.
	\end{equation}

	Next, we project the boundary edge by the closest-point projection onto the smooth extension $\M^e$, denoted by $\pi_\M$. Since $\gamma(u)\in\partial\M\subset\M^e$, we have
	\begin{equation}\label{eq:ass_4_guarantee_3}
	\pi_{\M}(\gamma(u))=\gamma(u), \qquad
	\mathrm{D}\pi_{\M}(\gamma(u))\,\gamma'(u)=\gamma'(u).
	\end{equation}
	The smoothness of $\pi_{\M}$ and the left identity of \eqref{eq:ass_4_guarantee_3} give the position estimate
	\[
	\big|\pi_{\M}(c(u))-\gamma(u)\big|
	\leq C|c(u)-\gamma(u)|\leq C\ell^2\leq Ch^2.
	\]
	For the derivative, the chain rule gives
	\[
	\frac{\d}{\d u}\pi_{\M}(c(u))
	=\mathrm{D}\pi_{\M}(c(u))\,c'(u).
	\]
	Thus, by the right identity of \eqref{eq:ass_4_guarantee_3}
	\begin{align*}
		\left|\frac{\d}{\d u}\pi_{\M}(c)-\gamma'(u)\right|
		&\leq\big\|\mathrm{D}\pi_{\M}(c)
			-\mathrm{D}\pi_{\M}(\gamma)\big\|\,|c'(u)|
		+\big\|\mathrm{D}\pi_{\M}(\gamma)\big\|
		|c'(u)-\gamma'(u)|.
	\end{align*}
	Applying the smoothness of $\mathrm{D} \pi_\M$ and the mean value theorem to the first term, and \eqref{eq:ass_4_guarantee_4} to the second term, we obtain
	\[
	\left|\frac{\d}{\d u}\pi_{\M}(c)-\gamma'(u)\right| \leq Ch^2 + C'h \leq C''h.
	\]

	Finally, by Assumption~\ref{ass:source} (iv), applying the above arguments to all arcs $\gamma$ that join the endpoints of the boundary edges of $\Mh$ concludes the proof.
\end{proof}

This gives a proper correction map for the surface $\M$.

\begin{lemma} \label{lem:bdry_correct}
	Let $\M$ be a compact smooth surface with $\partial \M \neq \varnothing$ and let $\Mh$ be the triangular mesh approximating $\M$ satisfying Assumption~\ref{ass:source}. For sufficiently small $h$,
	\begin{enumerate}
		\item the closest-point projection of the boundary
		\[
		\pi_{\partial \M} : \partial \big(\pi_\M(\Mh) \big)\to \partial \M
		\]
		is a homeomorphism.
		\item there exists $r_h$ such that $p_h = r_h \circ \pi_\M$ satisfies
		\begin{equation}\label{eq:correction}
		  \|p_h-\mathrm{id}\|_{L^\infty(\Mh)}\leq Ch^2,
		  \qquad |J_{p_h}-1|\leq Ch^2.
		\end{equation}
	\end{enumerate}
\end{lemma}

\begin{proof}
We parameterize a neighborhood $U$ of $\partial \M$ in the smooth extension $\M^e$ by
\[
\Psi: \left\{ (s,t)\,|\, |t| \leq \delta \right\} \to U \subset \M^e,
\]
where $s$ is the arclength coordinate of $\pi_{\partial\M}(x)$ and $t$ is the signed distance to $\partial\M$, positive inside $\M$. It is a diffeomorphism for sufficiently small $\delta > 0$. Applying $\Phi := \Psi^{-1}\circ\pi_{\M}$ to the arc $\gamma$ and the chord $c$, the arc $\gamma$ and the lifted edge $\pi_{\M}(c)$ in the coordinates $(s,t)$ are
\begin{align*}
	\Phi(\gamma(u)) = (s_0 + u, 0), \qquad
	\Phi(c(u))=(s(u),t(u)),
\end{align*}
where $s_0$ is the coordinate of $\gamma(0)$. Since $\pi_{\partial\M}$ is the projection $(s,t)\mapsto(s,0)$ in these coordinates, to show its bijectivity, it suffices to show that the lifted boundary is a graph $t=\beta_h(s)$ over $\partial\M$.

By the smoothness of $\Phi$ and \eqref{eq:boundary_estimate_3},
\[
|t(u)| \leq
\big\|(s(u),t(u))-(s_0+u,0)\big\|\le C\|c(u)-\gamma(u)\|\le C h^2 .
\]
For the derivatives, the smoothness of $\Phi$, \eqref{eq:boundary_estimate_3}, and \eqref{eq:ass_4_guarantee_4} give
\begin{align*}
	\|(s'(u), t'(u)) - (1,0)\| &=
	\Big\|\frac{\mathrm d}{\mathrm du}\Phi(c(u))-\frac{\mathrm d}{\mathrm du}\Phi(\gamma(u))\Big\|\\
	&=\Big\|\big(\mathrm D\Phi(c(u))-\mathrm D\Phi(\gamma(u))\big)\,c'(u)+\mathrm D\Phi(\gamma(u))\,\big(c'(u)-\gamma'(u)\big)\Big\|\\
	&\leq Ch^2 + Ch \leq Ch
\end{align*}
and thus
\[
|s'(u)-1|\leq Ch, \qquad |t'(u)|\leq Ch.
\]
Hence, for sufficiently small $h$, $s(u)$ is strictly increasing, and its inverse is well defined. Each lifted boundary edge is locally a graph $t=\beta_h(s)$ satisfying
\begin{equation}\label{eq:cutoff_boundary_2}
|\beta_h(s)|\leq Ch^2, \qquad |\beta_h'(s)|=|t'(u)/s'(u)|\leq Ch.
\end{equation}
Moreover, since $\pi_{\M}$ fixes the vertices $\gamma(0)$ and $\gamma(\ell)$, we have $s(0)=s_0$ and $s(\ell)=s_0+\ell$, so each lifted edge is the graph of $\beta_h$ over its own arc. By Assumption~\ref{ass:source} (iv), $\partial\big(\pi_\M(\Mh)\big)$ is the graph of $\beta_h$ over $\partial\M$, and hence $\pi_{\partial\M}$ is a homeomorphism.

Next, we construct the correction map $r_h$ in the coordinates $(s,t)$, which pushes the graph $t=\beta_h(s)$ down to $t=0$. To restrict this correction to a neighborhood of the boundary, we choose a smooth, compactly supported cutoff function $\zeta$ such that $0\leq\zeta\leq1$ and
\[
\zeta(t)=
\begin{cases}
	1 & \text{for }t\text{ near }0,\\
	0 & \text{for }|t|\text{ near }\delta.
\end{cases}
\]
Hence, the correction in the coordinates $(s,t)$ is
\begin{equation}\label{eq:cutoff_boundary}
	P_h(s,t)=\big(s,t-\zeta(t)\beta_h(s)\big),
\end{equation}
which maps $(s,\beta_h(s))$ to $(s,0)$ and is the identity outside $U$. Thus, the position error satisfies
\[
\|P_h-\mathrm{id}\|_{L^\infty}
\leq\|\beta_h\|_{L^\infty}\leq Ch^2.
\]

For the Jacobian estimate, the Jacobian matrix of $P_h$ is
\begin{equation}\label{eq:cutoff_boundary_Jacobian}
\mathrm{D}P_h=
\begin{bmatrix}
	1 & 0\\
	-\zeta(t)\beta_h'(s) & 1-\zeta'(t)\beta_h(s)
\end{bmatrix}.
\end{equation}
Write the area element near $\partial \M$ as
\[
\d A=w(s,t)\,\d s\,\d t,
\]
where $w$ is smooth and bounded away from zero. The correction changes the area element to
\[
\d A'=w\big(s,t-\zeta(t)\beta_h(s)\big)
\det(\mathrm{D}P_h)\,\d s\,\d t.
\]
Thus, substituting \eqref{eq:cutoff_boundary_Jacobian} into the area Jacobian gives
\begin{align*}
	J_{P_h}
	=\frac{w\big(s,t-\zeta(t)\beta_h(s)\big)}{w(s,t)}
	\det(\mathrm{D}P_h).
\end{align*}
By the mean value theorem and $|\beta_h|\leq Ch^2$, the first factor is bounded by
\[
\left|w\big(s,t-\zeta(t)\beta_h(s)\big)-w(s,t)\right|
=\left|\partial_t w(s,\xi)\,\zeta(t)\beta_h(s)\right|\leq Ch^2,
\]
where $\xi$ lies between $t$ and $t-\zeta(t)\beta_h(s)$. In addition, by \eqref{eq:cutoff_boundary_Jacobian}, the second factor is bounded by
\[
|\det(\mathrm{D}P_h) - 1| = |\zeta'(t)\beta_h(s)|\leq Ch^2.
\]
Since $w$ is bounded away from zero, we obtain
\begin{equation}\label{eq:cutoff_boundary_1}
	1- C''h^2 \leq \bigg( \frac{w - Ch^2}{w} \bigg)  \big( 1 - Ch^2 \big) \leq
	J_{P_h} \leq \bigg( \frac{w + Ch^2}{w} \bigg)  \big( 1 + Ch^2 \big) \leq 1 + C'h^2.
\end{equation}

Consequently, we let $r_h = \Psi\circ P_h \circ \Psi^{-1}$ on $U$, extended by the identity outside $U$, and $p_h=r_h\circ\pi_{\M}$. For $x\in U$ and $z=\Psi^{-1}(x)$, by \eqref{eq:cutoff_boundary_2},
\[
|r_h(x)-x| = |\Psi(P_h(z))-\Psi(z)|
\leq\|\mathrm{D}\Psi \| \, |P_h(z) - z|=\|\mathrm D\Psi\|\,\zeta|\beta_h|\le Ch^2.
\]
Therefore, combining the estimates with Lemma~\ref{lem:thin_triangle},
\[
|p_h(x)-x|\leq|r_h(\pi_{\M}(x))-\pi_{\M}(x)|+|\pi_{\M}(x)-x|\leq Ch^2.
\]
Finally, by \eqref{eq:cutoff_boundary_1},
\[
|J_{r_h} \circ \Psi - 1|  = |J_{P_h} -  1| \leq Ch^2,
\]
and thus, by the chain rule and Lemma~\ref{lem:area_change},
\begin{equation}\label{eq:cutoff_boundary_3}
|J_{p_h}(x) - 1| = |J_{r_h}(\pi_\M(x))\,J_{\pi_\M|_\tau}(x) - 1| \leq |Ch^2 + C'h^2 + C C'h^2| \leq C''h^2. \qedhere
\end{equation}
\end{proof}

\subsection{Admissible maps and energy comparison}
\label{sec:lift_def}
For the image mesh $\Nh:=f_h(\Mh)$, since we do not require the image triangles to be shape-regular, as mentioned in Section~\ref{sec:close}, we remove the shape-regular condition and use the following assumption.

\begin{assumption}[Admissible maps]\label{ass:admissible}
	We consider discrete maps $f_h:\Mh\to\Nh=f_h(\Mh)$ that are orientation-preserving piecewise-linear homeomorphisms such that
	\begin{itemize}
		\item[(i)] the image vertices lie on $\N$, and the image boundary vertices lie on $\partial\N$;
		\item[(ii)] $\pi_\N|_{\Nh}$ is a homeomorphism onto its image, which is $\N$ if $\partial\N=\varnothing$;
		\item[(iii)] if $\partial\N\neq\varnothing$, the arcs of $\partial\N$ joining the endpoints of the boundary edges of $\Nh$ cover $\partial\N$ exactly once;
		\item[(iv)] the image mesh size satisfies $k_h\leq\eta_h$, where
		\[
		\eta_h\to0, \qquad \frac{h}{\eta_h}\to0 \qquad\text{as }h\to0.
		\]
	\end{itemize}
\end{assumption}
To simplify the notation, we denote
\begin{equation} \label{eq:admissible_set}
	\A_h = \left\{ f_h:\Mh\to\Nh \;\middle|\; f_h \text{ satisfies Assumption~\ref{ass:admissible}}\right\}.
\end{equation}
Following the proof of Lemma~\ref{lem:bdry_correct} and replacing Lemma~\ref{lem:area_change} with Lemma~\ref{lem:thin_triangle}, we have the following corollary.

\begin{corollary} \label{cor:bdry_correct_map}
	Let $\Mh$ be the triangular mesh approximating $\M$ satisfying Assumption~\ref{ass:source}. Suppose $\partial \N \neq \varnothing$ and $f_h \in \A_h$ as in \eqref{eq:admissible_set}. For sufficiently small $h$,
	\begin{enumerate}
		\item the closest-point projection of the boundary
		\[
		\pi_{\partial \N} : \partial \big(\pi_\N(\Nh) \big)\to \partial \N
		\]
		is a homeomorphism.
		\item there exists $s_h$ such that $q_h = s_h \circ \pi_\N$ satisfies
		\begin{equation}\label{eq:target_estimate}
			\|q_h-\mathrm{id}\|_{L^\infty(\Nh)}\leq Ck_h^2,
			\qquad 0\leq J_{q_h}\leq 1+Ck_h^2.
		\end{equation}
	\end{enumerate}
\end{corollary}

\begin{proof}
This follows from the proof of Lemma~\ref{lem:bdry_correct} on $\N$, with $h$ replaced by $k_h\leq\eta_h\to0$. The only difference is \eqref{eq:cutoff_boundary_3}, which requires shape regularity to use Lemma~\ref{lem:area_change}. We replace it by Lemma~\ref{lem:thin_triangle}, which gives
\[
 0\leq J_{\pi_\N|_\tau}(y)\leq1+Ck_h^2
\]
for $y\in\tau\subset\Nh$. Therefore,
\begin{align*}
	0\leq J_{q_h}(y)&=J_{s_h}(\pi_\N(y))\,J_{\pi_\N|_\tau}(y)\leq(1+Ck_h^2)^2\leq1+C'k_h^2. \qedhere
\end{align*}
\end{proof}

As a result, we lift the simplicial map $f_h$ to a homeomorphism by
\begin{equation}  \label{eq:lift}
	\mathrm{L}_h f_h
	:= q_h \circ f_h \circ p_h^{-1} \in\mathrm{Homeo}^+(\M,\N),
\end{equation}
where
\[
p_h = r_h \circ \pi_\M, \qquad q_h = s_h \circ \pi_\N,
\]
constructed by Lemma~\ref{lem:bdry_correct} and Corollary~\ref{cor:bdry_correct_map}, whereas if $\partial \M = \varnothing$, we set $r_h = s_h = \mathrm{id}$; then \eqref{eq:correction} and \eqref{eq:target_estimate} follow directly from Lemmas~\ref{lem:area_change} and~\ref{lem:thin_triangle}.

\begin{lemma} \label{lem:one_side_compare}
Let $\M$ be a compact smooth surface and let $\Mh$ be the triangular mesh approximating $\M$ satisfying Assumption~\ref{ass:source}. Suppose $f_h \in \A_h$ as in \eqref{eq:admissible_set}. For sufficiently small $h$,
\[
\widebar{\E_\rS}(\mathrm{L}_h f_h)
\leq \Big(1 + C(h^2 + \eta_h^2)\Big) \,E_\rS(f_h),
\]
where $C$ is independent of $f_h$.
\end{lemma}

\begin{proof}
For almost every $x\in\Mh$,  since $\mathrm{L}_hf_h\circ p_h=q_h\circ f_h$, the chain rule gives
\begin{equation} \label{eq:one_side_compare_1}
J_{\mathrm{L}_h f_h} \big( p_h(x) \big)  = \frac{J_{q_h} \big( f_h(x) \big) \, J_{f_h}(x)}{J_{p_h}(x)}.
\end{equation}
Substituting \eqref{eq:one_side_compare_1} into the energy of the lift and changing variables via $p_h$, we obtain
\begin{align*}
\widebar{\E_\rS}(\mathrm{L}_h f_h)
&= \int_{\M} J_{\mathrm{L}_h f_h}^2 \, \d A_\M
= \int_{\Mh} \frac{J_{q_h} \big( f_h(x) \big)^2 \, J_{f_h}(x)^2}{J_{p_h}(x)^2} \, J_{p_h}(x) \, \d A_{\Mh}\\
&= \int_{\Mh} \frac{J_{q_h} \big( f_h(x) \big)^2}{J_{p_h}(x)} \, J_{f_h}(x)^2 \, \d A_{\Mh}.
\end{align*}
By \eqref{eq:correction} and \eqref{eq:target_estimate}, we have the estimates
\[
|J_{p_h} - 1| \leq Ch^2, \qquad 0\leq J_{q_h}\leq 1+Ck_h^2.
\]
Thus, for sufficiently small $h$,
\[
\frac{J_{q_h} \big( f_h(x) \big)^2}{J_{p_h}(x)} \leq \frac{(1+Ck_h^2)^2}{1 - Ch^2} \leq (1 + Ck_h^2)(1+Ch^2) \leq 1 + C'(h^2 + k_h^2).
\]
Consequently, since $k_h\leq\eta_h$,
\begin{align*}
\widebar{\E_\rS}(\mathrm{L}_h f_h)
\leq \Big(1 + C(h^2 + k_h^2)\Big) \int_{\Mh} J_{f_h}(x)^2 \, \d A_{\Mh}
\leq \Big(1 + C(h^2 + \eta_h^2)\Big) \,E_\rS(f_h).
\end{align*}
\end{proof}

\subsection{Approximation by interpolation}
\label{sec:interp}

It remains to justify the approximation by interpolation in \eqref{eq:idea_limsup}. For a diffeomorphism $g$, we denote its piecewise-linear interpolant as $\mathrm I_hg$ and provide the following lemma.

\begin{lemma}\label{lem:recover_smooth_map}
Let $\{\Mh\}$ satisfy Assumption~\ref{ass:source}, and let $g:\M\to\N$ be an orientation-preserving diffeomorphism. Then
\begin{enumerate}
	\item For all sufficiently small $h$, the piecewise-linear interpolant $\mathrm I_hg$ satisfies Assumption~\ref{ass:admissible}, i.e., $\mathrm I_hg \in \A_h$.
	\item As $h\to0$,
	\[
	d_H(\mathrm L_h\mathrm I_hg,g)\to0,
	\qquad E_\rS(\mathrm I_hg)\to\E_\rS(g).
	\]
\end{enumerate}
\end{lemma}

\begin{proof}
\cite[Lemmas~13 and~14]{LiKY26} have shown that, for sufficiently small $h$, the interpolant $\mathrm{I}_h g$ is an orientation-preserving piecewise-linear homeomorphism, $\Nh$ satisfies \cite[Assumption~7]{LiKY26} with respect to $\N$, and $k_h \leq C_g h$.

To prove (1), it suffices to verify Assumption~\ref{ass:admissible} (i)--(iii) and to show that $k_h \leq C_g h$ implies $k_h \leq \eta_h$. In fact, \cite[Assumption~7]{LiKY26} contains Assumption~\ref{ass:admissible} (i) and (ii), and further requires the difference between $\pi_\N(\N_h)$ and $\N$ to lie within an $\O(h^2)$ strip of $\partial \N$. For (iii), $g$ maps $\partial\M$ homeomorphically onto $\partial\N$. Hence, for each boundary edge of $\Mh$, the arc of $\partial\M$ joining its endpoints is mapped to the arc of $\partial\N$ joining the endpoints of the image edge. Since the former arcs cover $\partial\M$ exactly once by Assumption~\ref{ass:source} (iv), the latter cover $\partial\N$ exactly once. For the mesh size,
\[
\frac{k_h}{\eta_h}\leq C_g\frac h{\eta_h}\to0,
\]
and hence we have $k_h\leq\eta_h$ for sufficiently small $h$.

For (2), the energy convergence follows from \cite[Theorem~10]{LiKY26}. Since $\mathrm{L}_h \mathrm{I}_h g \circ p_h=q_h\circ \mathrm{I}_h g$,
\[
d_H(\mathrm L_h\mathrm{I}_h g,g) =\sup_{x\in\Mh}\big|q_h(\mathrm{I}_h g(x))-g(p_h(x))\big|.
\]
For $x \in \Mh$, the triangle inequality gives
\begin{align*}
	&|q_h(\mathrm{I}_h g(x)) - g(p_h(x))| \\
	&\quad
	\leq |q_h(\mathrm{I}_h g(x)) -\mathrm{I}_h g(x)| + |\mathrm{I}_h g(x) - g(\pi_\M(x))| + | g(\pi_\M(x)) - g(p_h(x))|\\
	&\quad
	\leq\|q_h-\mathrm{id}\|_{L^\infty(\Nh)} + \|\mathrm{I}_h g - g\circ\pi_\M\|_{L^\infty(\Mh)}+\|g\circ\pi_\M-g\circ p_h\|_{L^\infty(\Mh)} .
\end{align*}
The first term is bounded by $Ck_h^2\leq C\eta_h^2$ by \eqref{eq:target_estimate} and the second is at most $C_gh^2$ by \cite[Lemma~9]{LiKY26}. For the third term, by the smoothness of $g$, Lemma~\ref{lem:area_change}, and \eqref{eq:correction},
\begin{align*}
|g(\pi_{\mathcal M}(x)) - g(p_h(x))| &\leq C |\pi_{\mathcal M}(x) - p_h(x)|\\
&\leq C \big( |\pi_{\mathcal M}(x) - x| + |x - p_h(x)| \big) \leq C'h^2. \qedhere
\end{align*}
\end{proof}

\section{\texorpdfstring{$\Gamma$}{Gamma}-convergence}
\label{sec:main}
We now establish the $\Gamma$-convergence of the stretch energy and provide the derived properties.

\subsection{The main theorem}\label{sec:gamma_thm}

The discrete energy $E_\rS$ is defined on piecewise-linear maps, whereas $\Gamma$-convergence compares functionals on one fixed space. We therefore extend $E_\rS$ by the lift:
\begin{equation}\label{eq:pushed_energy}
\widebar{E_\rS}_h(u) = \inf \left\{ E_\rS(f_h) \; \middle| \; f_h\in\A_h,\quad \mathrm{L}_h f_h = u \right\}, \qquad \inf \varnothing = +\infty,
\end{equation}
where $\A_h$ is defined in \eqref{eq:admissible_set}. Hence, $\widebar{E_\rS}_h(u)<\infty$ if $u$ is the lifting of a discrete map satisfying Assumption~\ref{ass:admissible}.

\begin{theorem}\label{thm:main}
Let $\{\Mh\}$ satisfy Assumption~\ref{ass:source}. Then $\widebar{E_\rS}_h$ and $\widebar{\E_\rS}$ in \eqref{eq:pushed_energy} and \eqref{eq:relaxed_energy} satisfy
\[
\widebar{E_\rS}_h \xrightarrow[\;h\to0\;]{\;\Gamma(d_H) \;} \widebar{\E_\rS}.
\]
Explicitly,
\begin{itemize}
\item[(i)] for every $u\in\mathrm{Homeo}^+(\M,\N)$ and every $u_h \to u$ in $\mathrm{Homeo}^+(\M,\N)$,
\[
\widebar{\E_\rS}(u) \leq \liminf_{h \to 0} \widebar{E_\rS}_h (u_h);
\]
\item[(ii)] for every $u \in \mathrm{Homeo}^+(\M,\N)$, there exists $u_h \to u$ in $\mathrm{Homeo}^+(\M,\N)$ such that
\[
\limsup_{h \to 0} \widebar{E_\rS}_h (u_h) \leq  \widebar{\E_\rS}(u).
\]
\end{itemize}
\end{theorem}

\begin{proof}
Let $u_h\to u$ in $\mathrm{Homeo}^+(\M,\N)$.
If $\liminf_{h\to0} \widebar{E_\rS}_h(u_h) = +\infty$, it holds directly. Otherwise, we select a subsequence $h_k\to0$ such that $\widebar{E_\rS}_{h_k} (u_{h_k}) < \infty$ and
\[
\lim_{k\to\infty}\widebar{E_\rS}_{h_k}(u_{h_k})=\liminf_{h\to0}\widebar{E_\rS}_h(u_h).
\]
By \eqref{eq:pushed_energy}, we choose admissible representatives $f_k$ at mesh size $h_k$ with
\[
\mathrm{L}_{h_k}f_k=u_{h_k},\qquad
E_\rS(f_k)\leq\widebar{E_\rS}_{h_k}(u_{h_k})+h_k.
\]
The lower semicontinuity in Lemma~\ref{lem:weak_conv_area} and Lemma~\ref{lem:one_side_compare} give
\begin{align*}
	\widebar{\E_\rS}(u) &\leq \liminf_{k \to \infty} \widebar{\E_\rS} (u_{h_k})
	= \liminf_{k \to \infty} \widebar{\E_\rS} (\mathrm{L}_{h_k} f_k)\\
	&\leq \liminf_{k \to \infty} \Big(1 + C(h_k^2+ \eta_{h_k}^2)\Big)E_\rS(f_k)\\
	&\leq \lim_{k \to \infty} \Big(1 + C(h_k^2+ \eta_{h_k}^2)\Big)\Big(\widebar{E_\rS}_{h_k}(u_{h_k})+h_k\Big)
	= \liminf_{h \to 0} \widebar{E_\rS}_h(u_h).
\end{align*}

If $\widebar{\E_\rS}(u) = \infty$, we choose $u_h = u$. Otherwise, we approximate $u$ by diffeomorphisms and each diffeomorphism by its interpolants. By Theorem~\ref{thm:energy_dense}, there exist diffeomorphisms $g_j$ such that
\begin{equation}\label{eq:liminf_1}
d_H(g_j,u) < \frac{1}{j}, \qquad \big|\E_\rS(g_j)-\widebar{\E_\rS}(u)\big| < \frac{1}{j}.
\end{equation}
By Lemma~\ref{lem:recover_smooth_map}, for each $j$, there exists sufficiently small $h_j>0$ such that, for all $0<h<h_j$, $\mathrm I_hg_j\in\A_h$ and
\begin{equation}\label{eq:liminf_2}
d_H(\mathrm L_h\mathrm I_hg_j,g_j) < \frac{1}{j},\qquad \big|E_\rS(\mathrm I_hg_j)-\E_\rS(g_j)\big| < \frac{1}{j}.
\end{equation}
Since these properties still hold when $h_j$ is decreased, we assume $h_1>h_2>\cdots$ and $h_j\to0$. We construct $u_h$ by
\[
u_h:=\mathrm L_h\mathrm I_hg_j\quad\text{for }h_{j+1}\leq h<h_j.
\]
Consequently, by \eqref{eq:liminf_1} and \eqref{eq:liminf_2},
\[
d_H(u_h,u)\leq d_H(u_h,g_j)+d_H(g_j,u) < \frac{2}{j},
\]
and by \eqref{eq:pushed_energy}
\begin{align*}
\widebar{E_\rS}_h(u_h)&\leq E_\rS(\mathrm I_hg_j)<\E_\rS(g_j)+\frac1j<\widebar{\E_\rS}(u) + \frac{2}{j}.
\end{align*}
Since $j \to \infty$ as $h \to0$ with $h_{j+1}\leq h<h_j$, we conclude the result.
\end{proof}

\subsection{Local minimizers}
\label{sec:local_min}
While the $\Gamma$-convergence does not control local minimizers in general \cite[Remark~1.23]{Brai02}, for discrete maps that are minimal in balls of a fixed radius, the local minimality can be preserved.

\begin{theorem}[Limits of local minimizers]\label{thm:local_min_limit}
Let $\rho>0$, and for every sufficiently small $h$ let $f_h \in \A_h$ in \eqref{eq:admissible_set} satisfy
\begin{equation}\label{eq:local_min}
	E_\rS(f_h)\leq E_\rS(g_h)\qquad\text{for all } g_h \in \A_h \text{ with }d_H(\mathrm{L}_hg_h,\mathrm{L}_hf_h)<\rho .
\end{equation}
If $\mathrm{L}_hf_h\to f$ in $\mathrm{Homeo}^+(\M,\N)$, then
\begin{itemize}
\item[(a)] $f$ minimizes $\widebar{\E_\rS}$ on the ball of radius $\rho$:
\[
\widebar{\E_\rS}(f)\leq\widebar{\E_\rS}(v)\qquad\text{for all }v\in\mathrm{Homeo}^+(\M,\N)\text{ with }d_H(v,f)<\rho .
\]
\item[(b)] $\widebar{\E_\rS}(f)<\infty$, and $E_\rS(f_h)\to\widebar{\E_\rS}(f)$ as $h\to0$.
\item[(c)] If $|\M|=|\N|$, then $f$ preserves area, and
$$E_\rS(f_h) \longrightarrow  |\M| =\min_{\mathrm{Homeo}^+(\M,\N)}\widebar{\E_\rS}.$$
\end{itemize}
\end{theorem}

\begin{proof}
Let $v \in \mathrm{Homeo}^+(\M,\N)$ satisfy $d_H(v,f)<\rho$ and $\widebar{\E_\rS}(v)<\infty$. By Theorem~\ref{thm:main}(ii), there exists $v_h\to v$ in $\mathrm{Homeo}^+(\M,\N)$ with
\[
\limsup_{h\to0}\widebar{E_\rS}_h(v_h)\leq\widebar{\E_\rS}(v)<\infty .
\]
In particular, $\widebar{E_\rS}_h(v_h)<\infty$ for sufficiently small $h$.
By \eqref{eq:pushed_energy}, we may choose $g_h \in \A_h$ such that
\[
\mathrm{L}_hg_h=v_h,\qquad E_\rS(g_h)\leq\widebar{E_\rS}_h(v_h)+h .
\]
By the triangle inequality,
\[
d_H(\mathrm{L}_hg_h, \mathrm{L}_hf_h)\leq d_H(v_h,v)+d_H(v,f)+d_H(f, \mathrm{L}_hf_h).
\]
The first and the last terms tend to zero, while the second is less than $\rho$. Hence $d_H(\mathrm{L}_hg_h,\mathrm{L}_hf_h)<\rho$ for sufficiently small $h$, and \eqref{eq:local_min} gives $E_\rS(f_h)\leq E_\rS(g_h)$.
Moreover, since $\widebar{E_\rS}_h(\mathrm{L}_hf_h)\leq E_\rS(f_h)$ by \eqref{eq:pushed_energy}, Theorem~\ref{thm:main}(i) gives
\begin{equation}\label{eq:local_min_chain}
\begin{aligned}
\widebar{\E_\rS}(f)&\leq\liminf_{h\to0}\widebar{E_\rS}_h(\mathrm{L}_hf_h)\leq\liminf_{h\to0}E_\rS(f_h)\\
&\leq\limsup_{h\to0}E_\rS(f_h)\leq\limsup_{h\to0}E_\rS(g_h)\leq\widebar{\E_\rS}(v),
\end{aligned}
\end{equation}
which proves (a).

For (b), by the approximation of homeomorphisms by diffeomorphisms~\cite{Mois77,Munk60}, we choose a diffeomorphism $v$ with $d_H(v,f)<\rho$ and $\widebar{\E_\rS}(v)<\infty$, so (a) gives $\widebar{\E_\rS}(f)<\infty$. Setting $v=f$ in \eqref{eq:local_min_chain} then gives $E_\rS(f_h)\to\widebar{\E_\rS}(f)$.

For (c), since $f$ minimizes $\widebar{\E_\rS}$ on the ball of radius $\rho$ around it, Lemma~\ref{lem:local_global} shows that $f$ preserves area and $\widebar{\E_\rS}(f)=|\M|$, and (b) gives $E_\rS(f_h)\to\widebar{\E_\rS}(f)$.
\end{proof}

Consequently, as long as the minimality of discrete local minimizers holds in balls of a fixed radius and their lifts converge, they converge to global minimizers of the continuum problem. In particular, when $|\M|=|\N|$, the limit is an area-preserving homeomorphism.

\subsection{Almost minimizers and perturbations}
\label{sec:perturb}

These results give the convergence of the minimal values.

\begin{corollary}  \label{cor:min_converge}
Let $\M$ be a compact smooth surface and let $\Mh$ be the triangular mesh approximating $\M$ satisfying Assumption~\ref{ass:source}. Assume $|\M| = |\N|$. Then, for sufficiently small $h$,
\[
\frac{|\M|}{1 + C(h^2 + \eta_h^2)} \leq \inf_{\A_h} E_\rS \leq |\M| + Ch^2 .
\]
In particular,
\[
\inf_{\A_h} E_\rS ~\longrightarrow \min_{\mathrm{Homeo}^+(\M,\N)} \widebar{\E_\rS} = |\M|.
\]
\end{corollary}

\begin{proof}
	By \eqref{eq:energy_gap} and Lemma~\ref{lem:one_side_compare}, we have
	\[
	|\M| \leq \widebar{\E_\rS} (\mathrm{L}_h f_h) \leq \big(1 + C(h^2 + \eta^2_h) \big) \, E_\rS(f_h),
	\]
	and taking the infimum over $f_h$ gives the lower bound.
	On the other hand, Moser's theorem \cite{Mose65, Bany1974} provides an orientation-preserving diffeomorphism $g$ with $J_g = 1$. Lemma~\ref{lem:recover_smooth_map} and \cite[Lemma~15]{LiKY26} ensure that, for sufficiently small $h$, the interpolant satisfies $\mathrm I_hg \in \A_h$ and
	\[
	E_\rS(\mathrm{I}_hg) \leq |\Mh| + Ch^2.
	\]
	In addition, the estimate~\eqref{eq:correction} gives $J_{p_h} \geq 1 - Ch^2$ and hence $|\Mh|\leq|\M|+Ch^2$. Therefore,
	\[
	\inf_{\A_h} E_\rS \leq E_\rS(\mathrm{I}_hg)  \leq |\M| + Ch^2,
	\]
	which gives the upper bound.
\end{proof}

This generalizes the result of \cite[Theorem~16]{LiKY26}, which only considers the space of diffeomorphisms. It also provides a bound on the area distortion of almost minimizers.

\begin{corollary}\label{cor:distortion}
	Let $\M$ be a compact smooth surface and let $\Mh$ be the triangular mesh approximating $\M$ satisfying Assumption~\ref{ass:source}. Suppose $f_h \in \A_h$ as in \eqref{eq:admissible_set} satisfies
	\begin{equation} \label{eq:almost_min}
		E_\rS(f_h) \leq \inf_{\A_h} E_\rS + \varepsilon_h,\qquad \varepsilon_h\geq0,\quad \varepsilon_h\to0 .
	\end{equation}
	If $|\M| = |\N|$, then
	\[
	\| \widehat J_{\mathrm{L}_hf_h} - 1 \|^2_{L^2(\M)} \leq \big( 1 + C(h^2 + \eta_h^2) \big) \, ( |\M| + Ch^2 + \varepsilon_h) - |\M|,
	\]
	and the right-hand side tends to zero as $h \to 0$.
\end{corollary}

\begin{proof}
	By \eqref{eq:energy_gap} and $|\M| = |\N|$,
	\[
	\| \widehat J_{\mathrm{L}_hf_h} - 1 \|^2_{L^2(\M)}  = \widebar{\E_\rS} (\mathrm{L}_hf_h) - |\M|.
	\]
	By Lemma~\ref{lem:one_side_compare} and \eqref{eq:almost_min}, we obtain
	\begin{align*}
		\| \widehat J_{\mathrm{L}_hf_h} - 1 \|^2_{L^2(\M)}  &\leq \big( 1 + C(h^2 + \eta_h^2) \big) E_\rS(f_h) - |\M|\\
		&\leq \big( 1 + C(h^2 + \eta_h^2) \big) \Big( \inf_{\A_h} E_\rS + \varepsilon_h\Big) - |\M|.
	\end{align*}
	Since $\inf_{\A_h} E_\rS \leq |\M| + Ch^2$ by Corollary~\ref{cor:min_converge}, we obtain the upper bound, which tends to zero because $\eta_h\to0$ and $\varepsilon_h\to0$.
\end{proof}

In applications such as surface registration, we may need to compute an area-preserving parameterization while matching landmark points. Therefore, the stretch energy may be combined with other terms, such as a penalty on the mismatch of landmarks. $\Gamma$-convergence is stable under continuous perturbations~\cite[Remark~1.7]{Brai02}.

\begin{corollary}
\label{cor:perturb}
Let $\{\Mh\}$ satisfy Assumption~\ref{ass:source}.
If $G:\mathrm{Homeo}^+(\M,\N)\to[0,\infty)$ is continuous with respect to $d_H$, then
\[
\widebar{E_\rS}_h+G\xrightarrow[\;h\to0\;]{\;\Gamma(d_H)\;}\widebar{\E_\rS}+G.
\]
\end{corollary}

\begin{proof}
If $u_h\to u$ in $\mathrm{Homeo}^+(\M,\N)$, then $G(u_h)\to G(u)$, so Theorem~\ref{thm:main}(i) gives
\[
\liminf_{h\to0}\big(\widebar{E_\rS}_h(u_h)+G(u_h)\big)=\liminf_{h\to0}\widebar{E_\rS}_h(u_h)+G(u)\geq\widebar{\E_\rS}(u)+G(u).
\]
For the recovery sequence $\{u_h\}$ of Theorem~\ref{thm:main}(ii), the same continuity gives
\[
\limsup_{h\to0}\big(\widebar{E_\rS}_h(u_h)+G(u_h)\big)=\limsup_{h\to0}\widebar{E_\rS}_h(u_h)+G(u)\leq\widebar{\E_\rS}(u)+G(u).\qedhere
\]
\end{proof}

For instance, a landmark penalty takes the form
\[
G(u):=\lambda\sum_{i=1}^n|u(x_i)-y_i|^2,
\]
with landmarks $x_1,\ldots,x_n\in\M$, targets $y_1,\ldots,y_n\in\N$, and a weight $\lambda>0$. This term penalizes the mismatch between $u(x_i)$ and $y_i$. Let $a_i = u(x_i) - y_i$ and $b_i = v(x_i) - y_i$.
Since
\begin{align*}
|G(u)-G(v)| &\leq \lambda \sum_{i = 1}^n \Big| |a_i|^2 - |b_i|^2 \Big|
= \lambda \sum_{i = 1}^n \Big| (|a_i| + |b_i|) \, (|a_i| - |b_i|) \Big|
\end{align*}
and
\[
\big||a_i| - |b_i|\big| \leq |a_i - b_i| \leq d_H(u, v), \qquad
|a_i| + |b_i| \leq 2\operatorname{diam}(\N),
\]
where $\operatorname{diam}(\N)=\max_{p,q\in\N}|p-q|$ is the diameter of $\N$, we have
\[
|G(u)-G(v)|\leq2\lambda n\operatorname{diam}(\N)\,d_H(u,v),
\]
so $G$ is continuous on $\mathrm{Homeo}^+(\M,\N)$, and hence Corollary~\ref{cor:perturb} applies.

\small
\bibliographystyle{abbrvurl}
\bibliography{reference}

\end{document}